\documentclass[12pt]{article}
\usepackage{amssymb,amsmath}
\usepackage{latexsym,bm}
\usepackage{dsfont}

\usepackage{graphicx,amssymb,mathrsfs,amsmath,latexsym,amsfonts,amsthm}
\usepackage{amsmath,amssymb,amsthm}
\usepackage{latexsym}
\usepackage{amssymb}
\usepackage{stmaryrd}
 \usepackage{float}
\usepackage{psfrag}
\usepackage{xcolor}
\usepackage{enumerate}
\usepackage{manfnt,mathrsfs}

\makeatletter 
\@addtoreset{equation}{section}
\makeatother
\newtheorem{theorem}{Theorem}

\newtheorem{corollary}[theorem]{Corollary}

\newtheorem{lemma}[theorem]{Lemma}

\newtheorem{problem}[theorem]{Problem}
\def\adots{\mathinner{\mkern2mu\raise0pt\hbox{.}  
\mkern2mu\raise4pt\hbox{.}\mkern1mu
\raise7pt\vbox{\kern7pt\hbox{.}}\mkern1mu}}

\def\F{\mathscr{F}_k}

\newcommand{\C}{{\Bbb C}}

\def\tu{{\rm Tur{\'a}n\,}}
 \date{}
\begin{document}
\title{ \tu problems for digraphs  avoiding distinct walks of a
given length  with the same
endpoints}
\author{ Zejun Huang,\thanks{Institute of Mathmatics, Hunan University, Changsha  410082, P.R. China.  (mathzejun@gmail.com) } ~~Zhenhua Lyu,\thanks{College of Mathematics and Econometrics, Hunan University, Changsha  410082, P.R. China.  (lyuzhh@outlook.com)}~~ Pu Qiao\thanks{Department of
Mathematics, East China Normal University, Shanghai
200241, China. (235711gm@sina.com)}  } \maketitle

\begin{abstract}
 Let $n\ge 5$ and $k\ge 4$ be positive integers.
We determine the maximum size of digraphs of order $n$ that avoid distinct walks of
 length $k$ with the same
endpoints. We also characterize the extremal digraphs attaining this maximum number when $k\ge 5$.
\end{abstract}


\section{Introduction}

 \tu problems concern  the study of the maximum number,  called \tu number, of edges in graphs  containing no given subgraphs and the extremal graphs realizing that maximum. Mantel's theorem determines the maximum number of edges of triangle-free simple graphs as well as the unique graph attaining that maximum. Paul \tu  \cite{PT,PT2} generalized Mantel's theorem  by determining the maximum number of edges of $K_r$-free graphs on $n$ vertices and the unique graph  attaining that maximum, where $K_r$ denotes the complete graph on $r$ vertices.
\tu's theorem initiated the development of a major branch of graph theory, known as extremal graph theory \cite{BB,VN}.
Most of the previous results in extremal graph theory concern undirected graphs and only a few extremal problems on digraphs have been investigated; see \cite{BB,BES,BES2,BH,BS,JM}.
In this paper we study  extremal problems on digraphs.

We consider strict digraphs, i.e., digraphs without loops and parallel arcs.  For digraphs, we abbreviate directed walks and directed cycles as walks and cycles, respectively. The number of vertices in a digraph is called its {\it order} and the number of arcs its {\it size}. We   use $\overrightarrow{K}_r$ and $\overrightarrow{C}_r$ to denote the complete digraph and the directed cycle on $r$ vertices.

One natural \tu problem on digraphs is determining the maximum size of a $\overrightarrow{K}_r$-free strict digraph of a given order, which has been solved in \cite{JM}.

Note that the $k$-cycle is a  generalization  of the triangle  when we view a triangle as a 3-cycle in  undirected graphs. Another generalization of Mantel's Theorem is the \tu problem  for $k$-cycle-free  graphs.
However, this problem is difficult even for $C_4$-free graphs \cite{FK,TT}.
  An alternative direction on this problem is considering the orientations of $C_k$-free graphs. For example,  $C_4$ has the following orientations.

 \begin{figure}[H]
        \centering
        \includegraphics[width=1in]{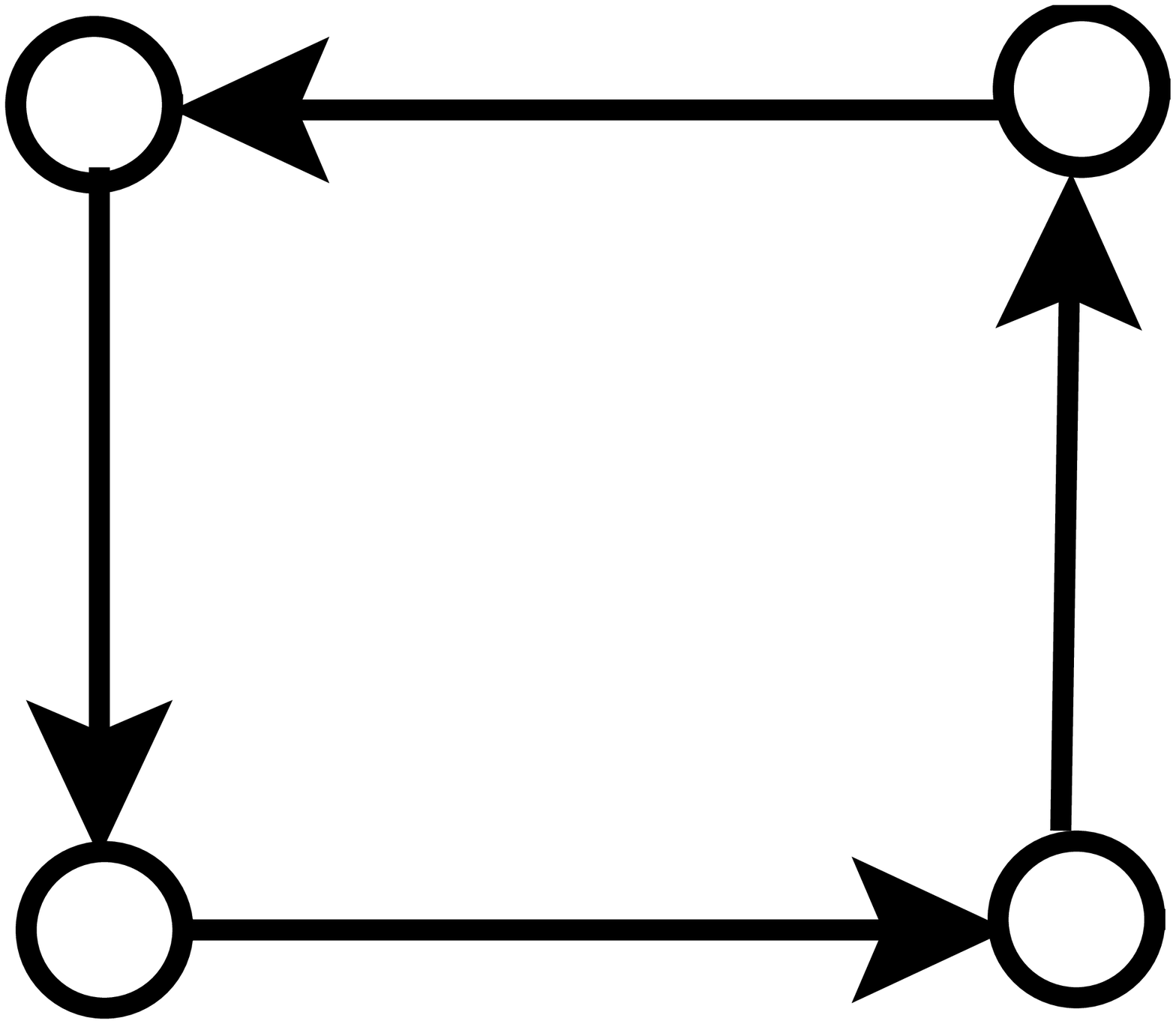}\hspace{0.8cm}
        \includegraphics[width=1in]{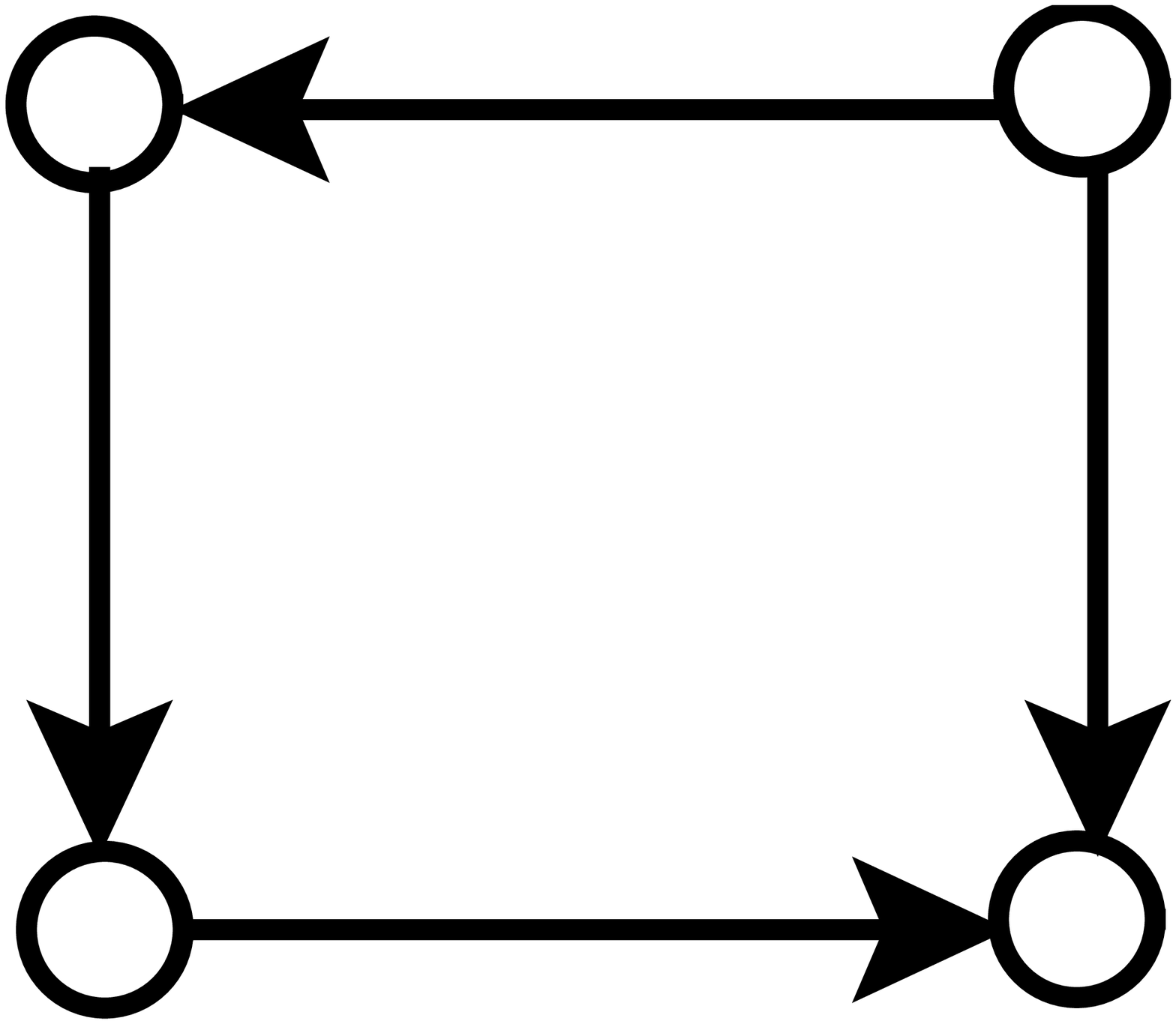}\hspace{0.8cm}
        \includegraphics[width=1in]{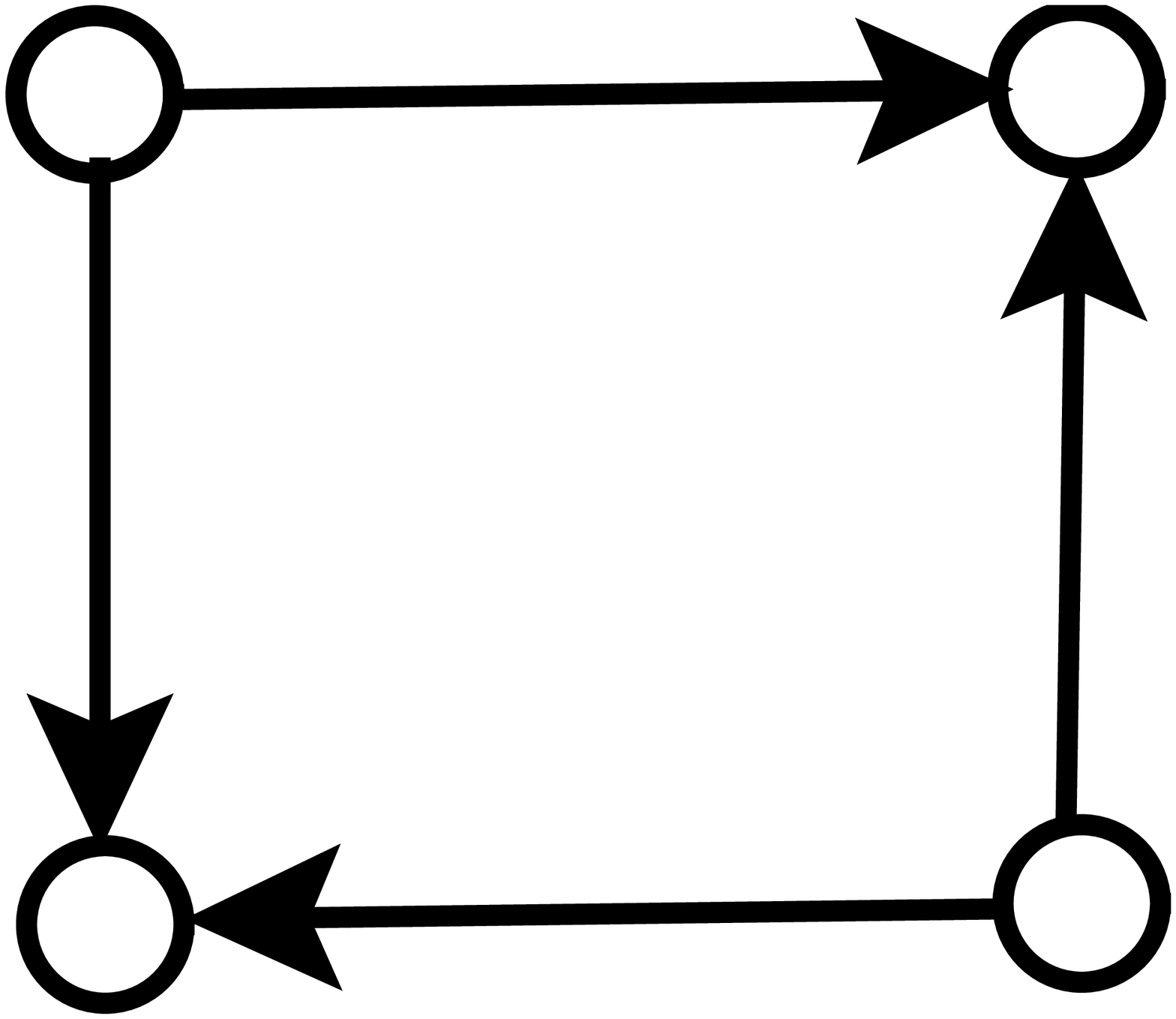}\hspace{0.8cm}
        \includegraphics[width=1in]{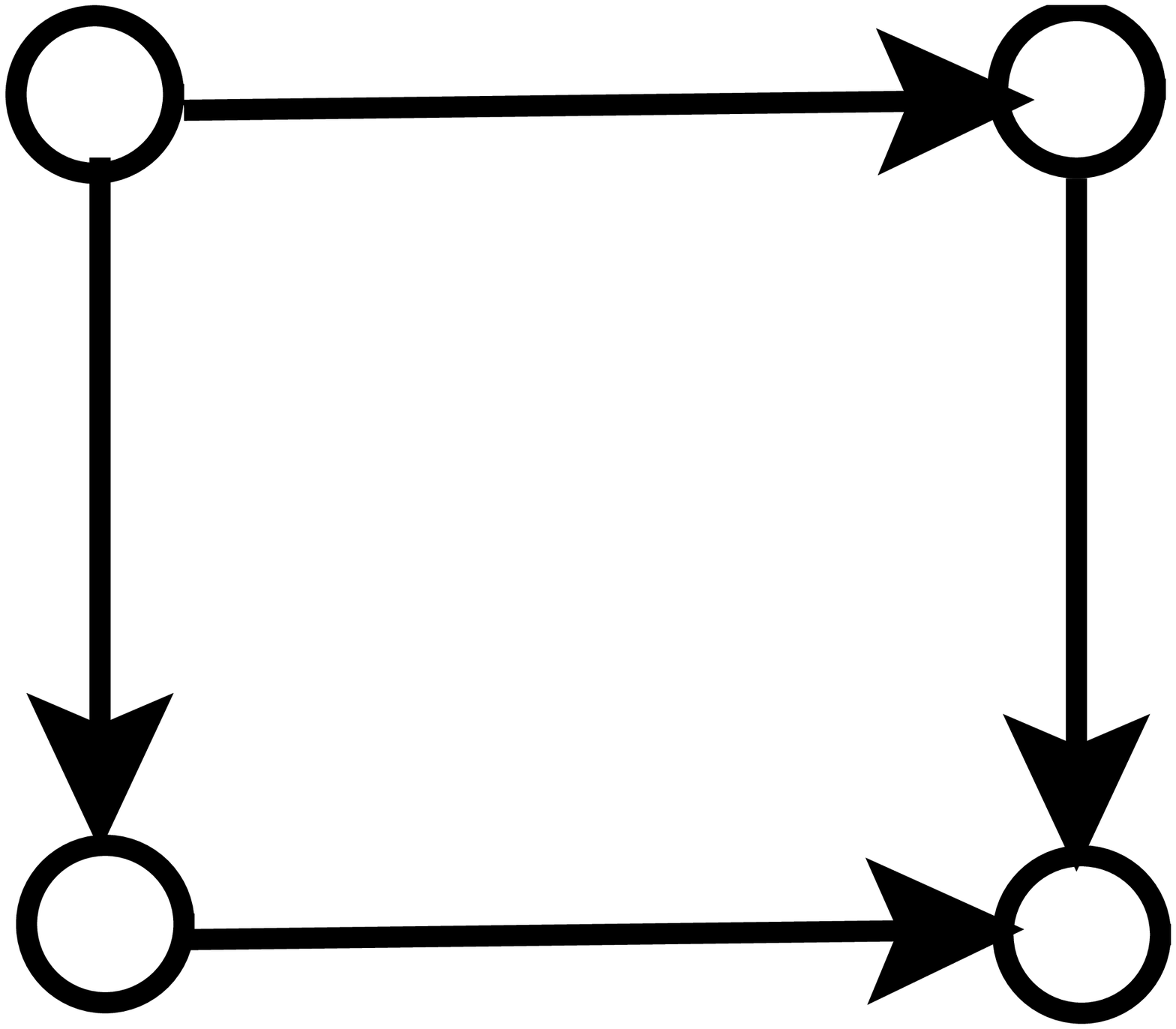}\\

       $C_4^{(1)}$\hspace{2.8cm}$C_4^{(2)}$\hspace{2.8cm}
       $C_4^{(3)}$\hspace{2.8cm}$C_4^{(4)}$ \end{figure}
\noindent It is clear that a graph is $C_4$-free if and only if any of its orientation contains no copy of the above four digraphs. Hence the \tu number for $C_4$-free graphs is equal to that for $\{C_2,C_4^{(1)},C_4^{(2)},C_4^{(3)},C_4^{(4)}\}$-free digraphs. For  $t\in\{1,2,3,4\}$, the \tu problem for $C_4^{(t)}$-free digraphs has independent interests; see \cite{BH}. We will consider a generalization of the \tu problem for $C_4^{(4)}$-free digraphs.

Given a positive integer $k$, we denote by
 $\mathscr{F}_k$ the family of  digraphs consisting of two different walks of length $k$ with the same initial vertex and the same terminal vertex, which have the following diagram
   \begin{figure}[H]
        \centering
        \includegraphics[width=4in]{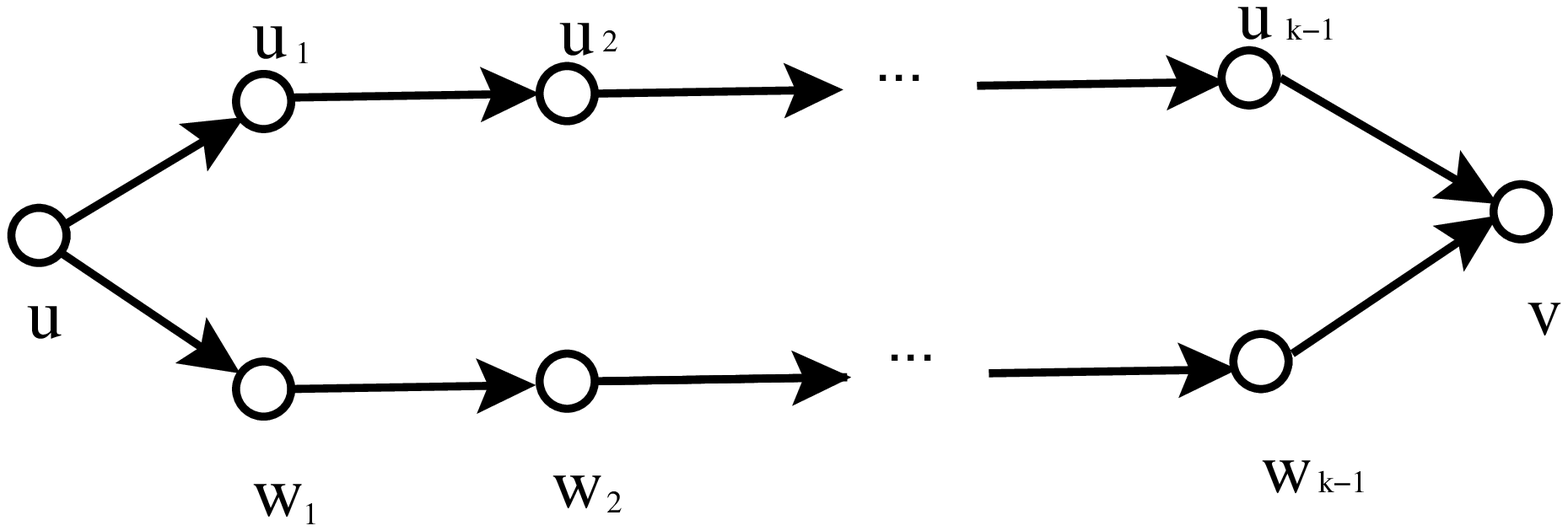}
\end{figure}
 \noindent where the vertices $u,v,u_1,u_2,\ldots, u_{k-1},w_1,w_2,\ldots,w_{k-1}$ can be  duplicate but $$(u_1,u_2,\ldots,u_{k-1})\ne (w_1,w_2,\ldots,w_{k-1}).$$

 We say a digraph $D$ is {\it $\F$-free} if $D$ contains no subgraph from $\F$.
 For any digraph $D$ on the vertices $1,2,\ldots,n$,  $D $ is $\F$-free if and only if there is at most one walk of length $k$ from $i$ to $j$ for every pair of vertices $i,j$. Let $ex(n,\F)$ be  the maximum size of $\F$-free strict digraphs of order $n$ and $Ex(n,\F)$ be the set of $\F$-free strict digraphs of order $n$ with size $ex(n,\F)$. We study the following   problem on strict digraphs.

\begin{problem}\label{pro1}
Given positive integers $n $ and $k$, determine $ex(n,\F)$ and  $Ex(n,\F)$.
\end{problem}

 When $k=1$, it is clear that $ex(n,\mathscr{F}_1)=n(n-1)$ and the unique extremal  digraph attaining $ex(n,\mathscr{F}_1)$ is the complete digraph of order $n$. When $k=2$,
 $\mathscr{F}_2$ consists of a unique digraph $C_4^{(4)}$.

In this paper, we always assume the order $n\ge 5$. We will determine $ex(n,\F)$ for $k\ge 4$ and characterize $Ex(n,\F)$  for $k\ge 5$. The paper is organized as follows. Section 2  presents our main result  Theorem \ref{thh2}, which determines $ex(n,\F)$ for $n\ge k+4\ge 8$ and characterizes $Ex(n,\F)$ for $n\ge k+5\ge 10$;  section 3 presents the characterization of $Ex(n,\F)$ for $k\ge 4$ and $n=k+2,k+3,k+4$;  section 4 presents the proof of Theorem \ref{thh2}; section 5 gives a discussion of the unsolved cases.

\section{Main result}
In order to present our main result, we need the follow notations and definitions.
Let $A$ be an $n\times n$ matrix and $\alpha=\{i_1,i_2,\ldots,i_k\}\subseteq \{1,2,\ldots,n\}$. We denote by   $A[\alpha]$  or $A[i_1,i_2,\ldots,i_k]$  the principal submatrix of $A$ lying on its $i_1$-th, $i_2$-th, $\ldots$, $i_k$-th rows and columns, and denote by $A(\alpha)$  or $A(i_1,i_2,\ldots,i_k)$  the principal submatrix of $A$ obtained by deleting its $i_1$-th, $i_2$-th, $\ldots$, $i_k$-th  rows and columns.

 Let $D=(\mathcal{V},\mathcal{A})$ be a digraph with vertex set $\mathcal{V}=\{v_1,v_2,\ldots,v_n\}$ and arc set $\mathcal{A}$. Its {\it adjacency matrix} $A_D=(a_{ij})$ is defined by
\begin{equation}\label{eqh1}
a_{ij}=\left\{\begin{array}{ll}
                 1,& (v_i,v_j)\in \mathcal{A};\\
                 0,&\textrm{otherwise}.\end{array}\right.
                 \end{equation}
Conversely, given an $n\times n$ 0-1 matrix $A=(a_{ij})$, we can define its digraph $D(A)=(\mathcal{V},\mathcal{A})$ on vertices $v_1,v_2,\ldots,v_n$ by (\ref{eqh1}), whose adjacency matrix is $A$.

Let $A=(a_{ij})$ and $B=(b_{ij})$ be matrices of order $m$ and $n$, respectively.  $A\otimes B=(a_{ij}B)$ is the tensor product of $A$ and $B$, whose order is $mn$. Denote by $J_{m,n}$ and $J_n$ the $m\times n$ and $n\times n$  matrices with all entries equal to one,
 $$T_n=\begin{bmatrix}
0&1&\cdots&1\\
&\ddots&\ddots&\vdots\\
&&0&1\\
&&&0
\end{bmatrix}$$
 the upper triangular tournament matrix of order $n$, and
 $$\Pi_{m, n}=J_m\otimes T_n=\begin{bmatrix}
T_{n}&\cdots&T_{n}\\
\vdots&\ddots&\vdots\\
T_{n}&\cdots&T_{n}
\end{bmatrix}.$$

Two matrices $A$ and $B$ are said to be permutation similar if there is a permutation matrix $P$ such that $B=PAP^T$, where $P^T$ denotes the transpose of $P$. For two digraphs $D_1$ and $D_2$,   $A_{D_1}$ and $A_{D_2}$ are permutation similar if and only if $D_1$ and $D_2$ are isomorphic.

A digraph of order $n$ is called a {\it transitive tournament} if its adjacency matrix is permutation similar to $T_n$.  Suppose $m$ and $t<n$ are nonnegative integers. We say a digraph of order $mn+t$ is an {\it $(m,n,t)$-completely transitive tournament} if its adjacency matrix is permutation similar to $\Pi_{m+1, n}(\alpha)$, where $\Pi_{m+1, n}(\alpha)$ is an $ (mn+t)\times (mn+t)$ principal submatrix of $\Pi_{m+1, n}$ with $\alpha\subseteq\{mn+1,mn+2,\ldots,mn+n\}$ and $|\alpha|=n-t$.  When $t=0$, we see that  a digraph of order $mn$  is an $(m,n,t)$-completely transitive  tournament if and only if its adjacency matrix is permutation similar to $\Pi_{m, n}$. Moreover, an $(m,n,t)$-completely transitive  tournament is a subgraph of the $(m+1,n,0)$-completely transitive  tournament.

Now we are ready to state our main result.

 \begin{theorem}\label{thh2}
 Let $n=sk+t$ with $s, k, t$ being nonnegative integers such that $t<k$. If $ n\ge k+4\ge 8$, then
 \begin{equation}\label{eq22}
 ex(n,\F)= {n\choose 2} -{s\choose 2}k-st.
  \end{equation}Moreover,  if $n\ge k+5\ge 10$,  then a digraph $D$ is in $Ex(n,\F)$ if and only if $D$ is an $(s,k,t)$-completely transitive tournament.
 \end{theorem}

We will also determine $Ex(n,\F)$ for $k\ge 4$ and $n=k+2,k+3,k+4$, while  $ex(n,\F)$ for $ n\le k+3$ and $Ex(n,\F)$ for $n\le k +1$ can be easily deduced from \cite{HZ1}.

From now on we deal with digraphs with no parallel arcs but allowing loops, and we use the same notations $\F$, $ex(n,\F)$ and $Ex(n,\F)$ for   digraphs allowing loops as for strict digraphs.  We will give solutions to Problem \ref{pro1} for   digraphs allowing loops. The same results for strict digraphs follow straightforward, since there is no loop in these extremal digraphs from $Ex(n,\F)$.

 \section{$ex(n,\F)$ and $Ex(n,\F)$ for $
 n\le k+4$}

For given integers $n$ and $k$, denote by $M_n\{0,1\}$ the set of $n\times n$ 0-1 matrices, $f(A)$ the number of ones in a 0-1 matrix $A$,
$$\Gamma(n,k)=\left\{A\in M_n\{0,1\}: A^k\in M_n\{0,1\}\right\},$$
 $$\theta(n,k)=\max_{A\in \Gamma(n,k)}f(A) \quad \textrm{and} \quad \Theta(n,k)=\left\{A\in \Gamma(n,k): f(A)=\theta(n,k)\right\}.$$
 Let $A\in M_n\{0,1\}$, $B$  an $m\times m$ principal submatrix of $A$, then it is clear that $A\in \Gamma(n,k)$ implies $B\in \Gamma(m,k)$. Moreover, given any $n\times n$ permutation matrix $P$, $A\in \Gamma(n,k)$ if and only if $P^TAP\in \Gamma(n,k)$.

To determine $\theta(n,k)$ and $\Theta(n,k)$ is an interesting problem posed by Zhan (see \cite[page 234]{ZH}), which has been partially solved by Wu \cite{WU}, Huang and Zhan \cite {HZ1}.

Given a digraph $D$, the $(i,j)$-entry of $(A_D)^k$ equals $t$ if and only if there are exactly $t$ distinct directed walks of length $k$ from vertex $v_i$ to vertex $v_j$ in $D$.
Hence, a digraph $D$ is $\F$-free if and only if its adjacency matrix $A_D$ is in $\Gamma(n,k)$. Moreover,
\begin{equation}\label{eq03}
\theta(n,k) =ex(n,\F).
\end{equation}
It should be noticed that (\ref{eq03}) is not necessarily true for strict digraphs.

For digraphs allowing loops, Huang and Zhan  determined $ex(n,\mathscr{F}_k)$ and $Ex(n,\mathscr{F}_k)$ for $k\ge n-1\ge 4$ as follows.
 \begin{theorem}[\cite{HZ1}]\label{thh1}
 Let $n,k$ be given integers such that
$k\ge n-1\ge 4.$ Then $ex(n,\F)=n(n-1)/2$ and a digraph $D$ is in $Ex(n,\F)$ if and only if $D$ is a transitive tournament of order $n$.
\end{theorem}

\noindent They also determine $ex(n, \mathscr{F}_{n-2})=n(n-1)/2-1$ for $n\ge 6$ and $ex(n, \mathscr{F}_{n-3})=n(n-1)/2-2$ for $n\ge 7$. Hence, when $k\ge 4$,
 $ex(n,\F)$ for $5\le n\le k+3$ and $Ex(n,\F)$   for $5\le n\le k+1$  have been determined.

In the following of this section, we will determine $ex(k+4,\F)$ and $Ex(n,\F)$   for $k\ge 4$ and $n=k+2,k+3,k+4$.
We need the following lemmas.

 \begin{lemma}\label{le3}
  Let $n \ge 3$, $p$ and $q$ be nonnegative integers such that $(p,q)\ne (0,0)$, and let $A\in M_{n}\{0,1\}$. If
 $$f(A(i))\le \frac{(n-1)(n-2)}{2}-p\frac{n-1}{2}-q \textrm{ for all } 1\le i\le n,$$ then
 \begin{equation}\label{eqle1}
 f(A)\le \frac{n(n-1)}{2}-p\frac{n+1}{2}-q-1.
 \end{equation}
 \end{lemma}
 \begin{proof}
Using the same idea as in the proof of \cite[Corollary 10]{HZ1}, we count the number of ones in the principal submatrices $A(1),\ldots,A(n)$.    Note that each diagonal entry of $A$ appears $n-1$ times and each off-diagonal entry of $A$ appears $n-2$ times in these submatrices. Suppose $A$ has $d$ nonzero diagonal entries.  Then
 \begin{eqnarray*}
 (n-1)d+(n-2)[f(A)-d]&=&\sum_{i=1}^nf(A(i))
  \le  n\left[\frac{(n-1)(n-2)}{2}-p\frac{n-1}{2}-q \right].
 \end{eqnarray*}
 It follows that
 $$f(A)\le \frac{n(n-1)}{2}-p\frac{n+1}{2}-q-\frac{p+2q}{n-2}-\frac{d}{n-2}.$$
Since $(p+2q)/(n-2)>0$, $d/(n-2)\ge 0$  and $f(A)$ is an integer, we get (\ref{eqle1}).
 \end{proof}
For the sake of convenience, we will always use $\{1,2,\ldots,n\}$ to denote the vertex set of a digraph $D$ of order $n$ and use the notation $i\rightarrow j$ to denote the arc $(i,j)$.
\begin{lemma}\label{le4}
 Let $m=k+t+s+1$ with $s\ge1$, $k\ge1$, $t\ge 3$ being integers, and let $x_{1},y_{1}\in \mathbb{R}^{k}$, $x_{2},y_{2}\in \mathbb{R}^{t}$, $x_{3},y_{3}\in \mathbb{R}^{s}$.  If
 $$
 (a_{ij})=\begin{bmatrix}
             0&J_{k,t}&J_{k,s}&x_{1}\\
             0&T_{t}&J_{t,s}&x_{2}\\
             0&0&0&x_{3}\\
             y_{1}^T&y_{2}^T&y_{3}^T&\alpha
\end{bmatrix}\in \Gamma(m,t+1)$$
and
$$\sum\limits_{i=1}^{3}[f(x_{i})+f(y_{i})]+\alpha\ge s+k+2,$$
  then
  $$\alpha=0,~~
  y_{1}=0, ~~x_{3}=0,\textrm{ and }
   a_{im}a_{mj}=0 \textrm{ for all }j\le i+2,  1\le i,j\le n.$$
 \end{lemma}
 \begin{proof}
 Denote   $A=(a_{ij})$.
First  we claim that $x_{3}=0$ and $y_{1}=0$. Otherwise suppose $x_3\ne 0$ or $y_1\ne 0$. Then $a_{im}=1$ for some $i\in \{k+t+1,\ldots,k+t+s$\} or $a_{mj}=1$ for some $j\in \{1,2,\ldots,k\}$. It follows that   $D(A)$  has two distinct walks  of length $t+1$ from $k$ to $m$ or from $m$ to $k+t+1$:
$$\left\{\begin{array}{l}
 k\rightarrow k+1 \rightarrow k+3 \rightarrow\cdots\rightarrow k+t\rightarrow i\rightarrow m,\\
 k\rightarrow k+2 \rightarrow k+3 \rightarrow\cdots\rightarrow k+t\rightarrow i\rightarrow m,\\
\end{array}\right.$$
$$\left\{\begin{array}{l}
 m\rightarrow j\rightarrow k+1\rightarrow k+3\rightarrow k+4\rightarrow \cdots\rightarrow k+t+1,\\
 m\rightarrow j\rightarrow k+1\rightarrow k+2\rightarrow k+4\rightarrow \cdots\rightarrow k+t+1,\\
\end{array}\right.$$
which contradicts  $A\in \Gamma(m,t+1)$.
Hence,  $x_{3}$ and $y_{1}$ are zero vectors.

Next we assert that $\alpha=0$. Otherwise,  $\alpha=1$.  Since $$\sum\limits_{i=1}^{3}(f(x_{i})+f(y_{i}))\ge s+k+1,$$ we have either
 $$\textrm{}\sum\limits_{i=1}^{3}f(x_{i})\ge k+1 \textrm{ or } \sum\limits_{i=1}^{3}f(y_{i})\ge s+1.$$ If $\sum\limits_{i=1}^{3}f(x_{i})\ge k+1$, then the last column of $A$ has at least two nonzero entries $a_{im}=a_{jm}=1$ with $1\le i<j\le k+t$. Hence  $D(A)$  has the following two distinct walks of length $t+1$ from $i$ to $m$:
$$\left\{\begin{array}{l}
 i\rightarrow m\rightarrow m\rightarrow m\rightarrow\cdots\rightarrow m,\\
 i\rightarrow j\rightarrow m\rightarrow m\rightarrow\cdots\rightarrow m.
\end{array}\right.$$
If $\sum\limits_{i=1}^{3}f(y_{i})\ge s+1$, then the last row of $A$ has at least two nonzero entries $a_{mi}=a_{mj}=1$ with $k+1\le i<j\le m$ and $i\le k+t$. It follows that $D(A)$ has the following two distinct walks of length $t+1$ from $m$ to $j$:
$$\left\{\begin{array}{l}
 m\rightarrow m\rightarrow \cdots\rightarrow m\rightarrow i\rightarrow j, \\
 m\rightarrow m\rightarrow   \cdots\rightarrow m\rightarrow m\rightarrow j.
\end{array}\right.$$
In both cases we get contradictions. Therefore, $\alpha=0$.

Next we claim $a_{im}a_{mj}=0$ for $j\le i+2$.  Otherwise suppose $a_{im}=a_{mj}=1$ with $1\le i,j\le m-1$ and $j\le i+2$. Since $x_{3}=y_{1}=0$, we have $i\le k+t$ and $j\ge k+1$.
We distinguish the following cases to find two distinct walks of length $t+1$ with the same endpoints in $D(A)$, which contradicts $A\in \Gamma(m,t+1)$.
If $i\le k$, then $j\le k+2$ and $D(A)$ has
$$\left\{\begin{array}{l}
 i\rightarrow k+1\rightarrow k+2\rightarrow \cdots\rightarrow k+t+1 ,\\
i\rightarrow m\rightarrow j\rightarrow k+3\rightarrow \cdots\rightarrow k+t+1.
\end{array}\right.$$
If $k<i\le k+t-2$, then $j\le k+t$ and $D(A)$ has
$$\left\{\begin{array}{l}
 k\rightarrow k+1\rightarrow k+2\rightarrow \cdots\rightarrow k+t\rightarrow k+t+1,\\
 k\rightarrow k+1\rightarrow
  \cdots\rightarrow i\rightarrow m\rightarrow j\rightarrow i+3\rightarrow \cdots\rightarrow k+t+1.
\end{array}\right.$$
If $i= k+t-1$, then $j\le k+t+1$ and $D(A)$ has
\begin{equation*}
\left\{\begin{array}{l}
 k\rightarrow k+1\rightarrow k+2\rightarrow \cdots\rightarrow i\rightarrow k+t\rightarrow k+t+1,\\
 k\rightarrow k+1\rightarrow k+2\rightarrow \cdots\rightarrow i\rightarrow m\rightarrow k+t+1,\texttt{ \it  if } j=k+t+1,\\
 k\rightarrow k+2\rightarrow \cdots\rightarrow i\rightarrow m\rightarrow j\rightarrow\cdots\rightarrow   k+t+1, \texttt{ \it if } j\le k+t.
\end{array}\right.
\end{equation*}
If $i= k+t$,  then $j\le k+t+2$ and $D(A)$ has
$$\left\{\begin{array}{l}
 k\rightarrow k+2\rightarrow k+3\rightarrow \cdots\rightarrow i\rightarrow m\rightarrow j,\\
 k\rightarrow k+1\rightarrow k+3\rightarrow \cdots\rightarrow i\rightarrow m\rightarrow j.
\end{array}\right.$$
 Therefore,  $a_{im}a_{mj}=0$ for all $j\le i+2$.
\end{proof}

\begin{corollary}\label{co6}
Let $x,y\in \mathbb{R}^{n-1}$ with $n\ge 6$. If
\begin{equation}\label{eqh2}
 \begin{bmatrix}
T_{n-1}&x\\y^T&\beta
\end{bmatrix}\in \Gamma(n,n-2),
\end{equation}
 and
 $$f(x)+f(y)+\beta=n-2,$$
 then one of the following holds:
\begin{itemize}
\item[(1)] $x=( 1,\ldots,1,0)^T,y=0$ and $\beta=0$;
\item[(2)] $y=(0,1,\ldots,1 )^T,x=0$ and $\beta=0$.
\end{itemize}
\end{corollary}
\begin{proof}
 Denote the matrix in (\ref{eqh2}) by $A=(a_{ij})$. Applying Lemma \ref{le4} with $k=s=1$, we have
\begin{equation}\label{eqh3}
\beta=a_{n1}=a_{n-1,n}=0,\textrm{ and } a_{in}a_{nj}=0 \textrm{ for all }j\le i+2.
\end{equation}

  We assert $f(x)=0$ or $f(y)=0$. Otherwise, assume that $a_{i_0n}$ is the last nonzero component in $x$, and $a_{nj_0}$ is the first nonzero component in $y$. Since $f(x)+f(y)=n-2\le i_0+n-1-(j_0-1)$, we have $j_0-i_0\le 2$, and $a_{i_0n}a_{nj_0}=0$ follows from (\ref{eqh3}), which contradicts the assumption that $a_{i_0n}a_{nj_0}=1$. Therefore, $x=0$ or $y=0$. It follows that either (1) or (2) holds.
 \end{proof}

Now we are ready to characterize $Ex(k+2,\F)$ and $Ex(k+3,\F)$ for $k\ge 4$.

 \begin{theorem}\label{le6}
 Let $n\ge 6$ be an integer. Then
 \begin{equation}\label{eqh4}
 ex(n,\mathscr{F}_{n-2})=\frac{n(n-1)}{2}-1.
 \end{equation}
 Moreover, a digraph $D$ is in $Ex(n,\mathscr{F}_{n-2})$ if and only if $A_D$ is permutation similar to \begin{equation*}\label{eq2}
 K_n\equiv\begin{bmatrix}
 T_{n-2}&J_{n-2,2}\\0&0
 \end{bmatrix}\textrm{\quad or \quad}
 K'_n\equiv\begin{bmatrix}
 0&J_{2,n-2}\\0&T_{n-2}
 \end{bmatrix}.
 \end{equation*}

 \end{theorem}
 \begin{proof}
 By   \cite[Corollary 10]{HZ1} we get (\ref{eqh4}). Suppose $D$ is a digraph in $Ex(n,\mathscr{F}_{n-2})$ and $A\equiv A_D$.  Applying Lemma \ref{le3}, there exists some $i\in\{1,2,\ldots,n\}$ such that  $f(A(i))\ge \frac{(n-1)(n-2)}{2}$. Since $A(i)\in \Gamma(n-1,n-2)$, applying Theorem \ref{thh1} we get $f(A(i))=\frac{(n-1)(n-2)}{2}$ and  $A(i)$ is permutation similar to $T_{n-1}$. Using permutation similarity if necessary, without loss of generality we assume $i=n$ and
$$A=\begin{bmatrix}
T_{n-1}&x\\y^T&\alpha
\end{bmatrix} $$
with $x,y\in \mathbb{R}^{n-1}$. It follows that

$$f(x)+f(y)+\alpha=f(A)-f(A(n))=ex(n,\mathscr{F}_{n-2})-\frac{(n-1)(n-2)}{2}=n-2.$$

\noindent Applying Corollary \ref{co6},  one of the following holds.
\begin{itemize}
\item[(1)] $x=( 1,\ldots,1,0)^T,y=0$ and $\alpha=0.$  Then $A=K_n$;
\item[(2)] $y=(0,1,\ldots,1 )^T,x=0$ and $\alpha=0.$  Then $PAP^T=K'_n$,  where
$$P=\begin{bmatrix}
 0&1\\
 I_{n-1}&0
 \end{bmatrix}.$$
\end{itemize}
Therefore, $A_D$ is permutation similar to $K_n$ or $K_n'$.

 Conversely, if the adjacency matrix  $A$ of a digraph $D$ is permutation similar to $K_n$ or $K'_n$, by direct computation we can verify  $f(A)=ex(n, \mathscr{F}_{n-2})$ and $A^{n-2}\in M_{n}\{0,1\}$. Hence $D\in  Ex(n,\mathscr{F}_{n-2})$.
 \end{proof}

\begin{theorem}\label{leh6}
 Let $n\ge 7$ be an integer. Then
 \begin{equation}\label{eqh5}
 ex(n,\mathscr{F}_{n-3})=\frac{n(n-1)}{2}-2.
 \end{equation}
 Moreover, a digraph $D$ is in $Ex(n,\mathscr{F}_{n-3})$ if and only if $A_D$ is permutation similar to
 $$
F_n\equiv\begin{bmatrix}
             0&J_{2,n-4}&J_{2,2}\\
             0&T_{n-4}&J_{n-4,2}\\
             0&0&0
\end{bmatrix}.$$
 \end{theorem}
 \begin{proof}
 From  \cite[Corollary 11]{HZ1}
 we get $(\ref{eqh5})$. Suppose $D\in Ex(n,\mathscr{F}_{n-3})$ and   $A\equiv A_D$.  Applying \ref{le3} we see  that $A$ contains a submatrix $A(i)$, say $A(n)$, such that $f(A(n))\ge \frac{(n-1)(n-2)}{2}-1$. By Theorem \ref{le6},
  $f(A(n))=\frac{(n-1)(n-2)}{2}-1$ and $A(n)$ is permutation similar to $K_{n-1}$ or $K'_{n-1}$.

First  we consider the case that $A(n)$ is permutation similar to $K_{n-1}$. Without loss of generality we can assume $A(n)=K_{n-1}$ and
 $$A=\begin{bmatrix}
0&J_{1,n-4}&J_{1,2}&x_{1}\\
0&T_{n-4}&J_{n-4,2}&x_{2}\\
0&0&0&x_{3}\\
y_{1}^T&y_{2}^T&y_{3}^T&\alpha
\end{bmatrix},$$
 where $x_{1},y_{1}\in \mathbb{R}$, $x_{2},y_{2}\in \mathbb{R}^{n-4}$, and $x_{3},y_{3}\in \mathbb{R}^{2}$.

 Let $x=(x^T_{1}, x_{2}^T, x_{3}^T) $ and $y=(y_{1}^T,y_{2}^T,y_{3}^T) $. Then
$$\alpha+f(x)+f(y)=f(A)-f(A(n))=n-2.$$
 Applying Lemma \ref{le4}, we get $x_{3}=0$, i.e., $a_{n-2,n}=a_{n-1,n}=0$.

 Let $i=n-2$ or $n-1$.  Then
 \begin{equation}\label{eqq37}
 f(A(i))=f(A)-(n-3)-a_{ni}=  \frac{(n-1)(n-2)}{2}-a_{ni}.
  \end{equation}
  On the other hand, since $A(i)\in \Gamma(n-1,n-3)$, by Theorem \ref{le6} we have
  \begin{equation}\label{eqh6}
  f(A(i))\le\frac{(n-1)(n-2)}{2}-1.
   \end{equation}
Combining (\ref{eqq37}) and (\ref{eqh6}) we get $a_{ni}=1$.

 Now applying Corollary \ref{co6} to $A(n-1)$  we have
  $y=(0,1,\ldots,1 ),x=0$, $\alpha=0,$  and
$$A= \begin{bmatrix}
0&J_{1,n-4}&J_{1,2}&0\\
0&T_{n-4}&J_{n-4,2}&0\\
0&0&0&0\\
0&J_{1,n-4}&J_{1,2}&0
\end{bmatrix}=P^T F_nP$$
where $$P=\begin{bmatrix}
 0&1\\I_{n-1}&0
 \end{bmatrix}.$$

Next suppose $A(n)$ is permutation similar to $K'_{n-1}$. Without loss of generality we can assume $A(n)=K'_{n-1}$ and
 $$A=\begin{bmatrix}
0&J_{2,n-4}&J_{2,1}&x_{1}\\
0&T_{n-4}&J_{n-4,1}&x_{2}\\
0&0&0&x_{3}\\
y_{1}^T&y_{2}^T&y_{3}^T&\alpha
\end{bmatrix},$$
 where $x_{1},y_{1}\in \mathbb{R}^2$, $x_{2},y_{2}\in \mathbb{R}^{n-4}$, and $x_{3},y_{3}\in \mathbb{R}$. Applying the same argument as above by counting $f(A(1))$, $f(A(2))$  and applying Corollary \ref{co6} to $A(1)$,  we  get $A=F_{n}$.

Conversely, if the adjacency matrix  $A$ of a digraph $D$ is permutation similar to $F_n$, by direct computation we can verify  $f(A)=ex(n, \mathscr{F}_{n-3})$ and $A^{n-3}\in M_{n}\{0,1\}$. Hence $D\in  Ex(n,\mathscr{F}_{n-3})$.\\
 \end{proof}
 \par

 Next we determine $ex(k+4,\F)$ for $k\ge 4$ and $Ex(k+4,\F)$ for $k\ge 5$.
 \begin{lemma}\label{le7}
Let $x,y\in \mathbb{R}^{n-1}$ with $n\ge 6$, and
\begin{equation*}
 A=\begin{bmatrix}
T_{n-1}&x\\y^T&\alpha
\end{bmatrix}.
\end{equation*}
\begin{itemize}
\item[(i)]
If $f(x)+f(y)+\alpha=n-2$, then $A\in \Gamma(n,n-1)$ if and only if \begin{equation}\label{eqh39}
\alpha=0, ~~ x=(a^T,0)^T,~~ y=(0,b^T)^T,
\end{equation}  where $a\in \mathbb{R}^s$, $b\in \mathbb{R}^{n-s-1}$ with $ s\in\{0,1,\ldots,n-2\}$. Here $s=0$ means $x=0$.
\item[(ii)]
If $f(x)+f(y)+\alpha=n-1$,  then $A\in \Gamma(n,n-1)$ if and only if \begin{equation}\label{eqh310}
\alpha=0, ~~x=(J_{1, s},0)^T, ~~y=(0,J_{1,n-s-1})^T,
\end{equation}
where $s\in\{0,1,\ldots,n-2\}$.
  \end{itemize}
\end{lemma}
\begin{proof}
{\it(i)}~ Suppose $A\in \Gamma(n,n-1)$. First we claim $\alpha=0$. Otherwise, since $f(x)+f(y)=n-3\ge 3$, we have either $f(x)\ge 2$ or $f(y)\ge 2$. If $f(x)\ge 2$, say,  $a_{in}=a_{jn}=1$ with $1\le i<j\le n-1$, then $D(A)$ has the following two distinct walks from $i$ to $n$ with the same length $n-1$:
$$\left\{\begin{array}{l}
 i\rightarrow n\rightarrow n\rightarrow n\rightarrow\cdots\rightarrow n,\\
 i\rightarrow j\rightarrow n\rightarrow n\rightarrow\cdots\rightarrow n.
\end{array}\right.$$
If $f(y)\ge 2$, say,  $a_{ni}=a_{nj}=1$ with $1\le i<j\le n-1$, then $D(A)$ has the following two distinct walks from $n$ to $j$ with the same length $n-1$:
$$\left\{\begin{array}{l}
 n\rightarrow n\rightarrow  \cdots\rightarrow n\rightarrow i\rightarrow j, \\
 n\rightarrow n\rightarrow   \cdots\rightarrow n\rightarrow n\rightarrow j.
\end{array}\right.$$
In both cases we get contradictions. Hence $\alpha=0$.

Next we claim
\begin{equation}\label{eqn311}
a_{in}a_{nj}=0 \textrm{~~~for all~~~}i\ge j.
 \end{equation}
 If $x=0$ or $y=0$, the claim is clear. Suppose  $x$, $y$ are nonzero, and there exist $i\ge j$  such that $a_{in}a_{nj}=1$. Then we have the following cases  and in each of these cases $D(A)$ has   two different walks of length $n-1$ with the same endpoints, which contradicts $A\in \Gamma(n,n-1)$.

{\it Case 1.} $i\le 2$. $D(A)$ has
 $$\left\{\begin{array}{l}
 1\rightarrow i\rightarrow n\rightarrow j\rightarrow i+2\rightarrow i+3\rightarrow\cdots \rightarrow n-1,\\
 1\rightarrow i\rightarrow n\rightarrow j\rightarrow i+1\rightarrow i+3\rightarrow\cdots \rightarrow n-1,\\
\end{array}\right.$$
where the arc $1\rightarrow i$ does not appear if $i=1$.

{\it Case 2.} $i=3$.  $D(A)$ has
 $$\left\{\begin{array}{l}
 1\rightarrow 3\rightarrow n\rightarrow j\rightarrow i+1\rightarrow\cdots \rightarrow n-1,\\
 1\rightarrow 2\rightarrow 3\rightarrow n\rightarrow j\rightarrow i+2\rightarrow\cdots \rightarrow n-1.\\
\end{array}\right.$$

{\it Case 3.} $4\le i\le n-1$.  $D(A)$ has
$$\left\{\begin{array}{l}
 1\rightarrow 3\rightarrow 4\rightarrow\cdots \rightarrow i\rightarrow n\rightarrow j\rightarrow i+1\rightarrow\cdots \rightarrow n-1,\\
 1\rightarrow 2\rightarrow 4\rightarrow\cdots \rightarrow i\rightarrow n\rightarrow j\rightarrow i+1\rightarrow\cdots \rightarrow n-1,\\
\end{array}\right.$$
where the walk $j\rightarrow i+1\rightarrow\cdots \rightarrow n-1$ does not appear if $i=n-1$.

   Let $a_{sn}$ be the last nonzero component in $x$ and $a_{nt}$ be the first nonzero component in $y$. Since $f(x)+f(y)=n-2$,  by (\ref{eqn311}) we have $t-s=1$ or 2. Hence  (\ref{eqh39}) holds.

Conversely, suppose $A$ satisfies (\ref{eqh39}). Let
\begin{equation*}\label{eqq12}
B =\begin{bmatrix}
T_{s}&J_{s,n-s-1}&J_{s,1}\\
0&T_{n-s-1}&0\\
0&J_{1,n-s-1}&0
\end{bmatrix}
\end{equation*}
with $0\le s\le n-1$.
To prove $A\in \Gamma(n,n-1)$, it suffices to verify
 $B  \in \Gamma(n,n-1),$  since $B\ge A$, where the notation $\ge$ is to be understood entrywise.

If $s=n-1$, then $B=T_{n}\in \Gamma(n,n-1)$.
If $s<n-1$, then
$$B =\begin{bmatrix}
T_{s}&J_{s,1}&J_{s,n-s-2}&J_{s,1}\\
0&0&J_{1,n-s-2}&0\\
0&0&T_{n-s-2}&0\\
0&1&J_{1,n-s-2}&0
\end{bmatrix} $$
 is permutation similar to
$$\begin{bmatrix}
T_{s}&J_{s,1}&J_{s,1}&J_{s,n-s-2}\\
0&0&1&J_{1,n-s-2}\\
0&0&0&J_{1,n-s-2}\\
0&0&0&T_{n-s-2}
\end{bmatrix}=T_{n}.$$
Therefore, $B \in \Gamma(n,n-1)$. This completes the proof for (i).

 (ii) For the sufficiency part, if (\ref{eqh310}) holds, then $A=B\in \Gamma(n,n-1)$.   For the necessity part,  let $a_{sn}$ be the last nonzero component in $x$ and $a_{nt}$ be the first nonzero component in $y$. Since $f(x)+f(y)=n-1$, applying the same arguments as above we get $\alpha=0$ and $t-s=1$. It follows that (\ref{eqh310}) holds.
 \end{proof}
\begin{theorem}\label{le8}
Let $n\ge 8$ be an integer. Then
\begin{equation}\label{eq316}
ex(n,\mathscr{F}_{n-4})=\frac{n(n-1)}{2}-4.
\end{equation}
 \end{theorem}
\begin{proof}
  Let $A$ be the adjacency matrix of any $\mathscr{F}_{n-4}$-free digraph  $D$ of order $n$. Then $A\in \Gamma(n,n-4)$ and  $A(i)\in \Gamma(n-1,n-4)$ for all $1\le i\le n$. Hence
   $$f(A(i))\le  ex(n-1,\mathscr{F}_{n-4})=\frac{(n-1)(n-2)}{2}-2$$
   for all $1\le i\le n$.   By Lemma \ref{le3}, we have
   \begin{equation}\label{eq313}
   f(A)\le \frac{n(n-1)}{2}-3.
    \end{equation}

   Suppose equality in (\ref{eq313}) holds.
    Then by Lemma \ref{le3}, $A$ contains a submatrix $A(i)$ with $\frac{(n-1)(n-2)}{2}-2$ nonzero entries.  Using permutation similarity if necessary, without loss of generality we assume $f(A(n))=\frac{(n-1)(n-2)}{2}-2$. Since  $A(n)\in \Gamma(n-1,n-4)$, by Theorem \ref{leh6}, we may further assume
\begin{equation}\label{eq311}
A=(a_{ij})=\begin{bmatrix}
             0&J_{2,n-5}&J_{2,2}&x_{1}\\
             0&T_{n-5}&J_{n-5,2}&x_{2}\\
             0&0&0&x_{3}\\
             y^{T}_{1}&y^{T}_{2}&y^{T}_{3}&\alpha
\end{bmatrix},
\end{equation}
where $x_1,x_3,y_1, y_3\in \mathbb{R}^{2}$, $x_2,y_2\in \mathbb{R}^{n-5}$.

Let $x=(x_1^T,x_2^T,x_3^T)$ and $y=(y_1^T,y_2^T,y_3^T)$. Then $$f(x)+f(y)+\alpha=f(A)-f(A(n))=n-2.$$
Applying Lemma \ref{le4} to $A$ we know $y_{1}=x_{3}=0$ and $\alpha=0$.

  Since $a_{n-1,n}=a_{n1}=0$ and $A(1,n-1)\in \Gamma(n-2,n-4)$,  by (\ref{eq311}) we have
\begin{eqnarray*}
f(A(1,n-1))&=&f(A)-2(n-5)-3-a_{1n}-a_{n1}-a_{n-1,n}-a_{n,n-1}\\
 &=&\frac{(n-2)(n-3)}{2}+1-a_{1n}-a_{n,n-1}\\
 &\le&\frac{(n-2)(n-3)}{2}-1
  \end{eqnarray*}
  where the inequality follows from Theorem \ref{le6}.
   Hence,
\begin{equation}\label{eq312}
a_{1n}=a_{n,n-1}=1.
 \end{equation}
 On the other hand, applying Corollary \ref{co6} to $A(1,n-1)$ we have either $x=0$ or $y=0$, which contradicts (\ref{eq312}).
 Hence, (\ref{eq313}) is a strict inequality and we have
\begin{equation}\label{eq314}
ex(n,\mathscr{F}_{n-4})\le \frac{n(n-1)}{2}-4.
\end{equation}

Now let $D$ be the digraph with adjacency matrix
$$
B=\begin{bmatrix}
             0&J_{3,n-5}&J_{3,2}\\
             0&T_{n-5}&J_{n-5,2}\\
             0&0&0\\
\end{bmatrix}.$$ By direct computation, we have
$$
B^{n-4}=\begin{bmatrix}
             0&0&J_{3,2}\\
             0&0&0\\
             0&0&0\\
\end{bmatrix},$$
and hence $D$ is $\mathscr{F}_{n-4}$-free.
Therefore,
\begin{equation}\label{eq315}
ex(n,\mathscr{F}_{n-4})\ge f(B)=\frac{n(n-1)}{2}-4.
\end{equation}
Combining (\ref{eq314}) and (\ref{eq315}) we get (\ref{eq316}).
\end{proof}

\begin{lemma}\label{le10}
 Let $k\ge 5$ and $s$ be positive integers,  let $x_i,y_i\in \mathbb{R}^k$ with  components from $\{0,1\}$  for $i=1,2,\ldots,s$, and let  $$A=(a_{ij})=\begin{bmatrix}
T_{k}&x_1&x_2&x_3&\cdots&x_s\\
y_1^T&0&1&0&\cdots&0\\
y_2^T&0&0&1&\ddots&\vdots\\
y_3^T&0&0&0&\ddots&0\\
\vdots&\vdots&\vdots&\vdots&\ddots&1\\
y_s^T&1&0&\cdots&\cdots&0
\end{bmatrix}$$.
\begin{itemize}
\item[(i)]If there is some  $i\in\{1,2,\ldots,s\}$ such that $f(x_i)\ge 3$ or $f(y_{i})\ge 3$, then $A \notin \Gamma(k+s,p)$ for any integer $p\ge 2$.
 \item[(ii)]If $s=2$ and there is some  $i\in\{1,2\}$ such that $f(x_i)=f(y_{i})= 2$, then $A \notin \Gamma(k+s,p)$   for any integer $p\ge 5$.
\end{itemize}

 \end{lemma}
\begin{proof}
{\it (i)}~~If there is some $t$ such that  $f(x_t)\ge 3$,  then we have $a_{i,k+t}=a_{j,k+t}=a_{m,k+t}=1$ for $1\le i<j<m\le k$. Without loss of generality we assume $t=1$. For any $p\in \{2,3,\ldots,k\}$, we can find two  distinct walks of length $p$ between the same endpoints in the following walks in $D(A)$.
$$\left\{\begin{array}{l}
 i\rightarrow j\rightarrow k+1\rightarrow k+2\rightarrow\cdots\rightarrow k+s \rightarrow k+1\rightarrow k+2\rightarrow\cdots \\
 i\rightarrow m\rightarrow k+1\rightarrow k+2\rightarrow\cdots\rightarrow k+s \rightarrow k+1\rightarrow k+2\rightarrow\cdots
\end{array}\right.$$

If  there is some $t$, say $t=1$, such that  $f(y_t)\ge 3$, then we have $a_{k+1,i}=a_{k+1,j}=a_{k+1,m}=1$ for $1\le i<j<m\le k$. For any $p\in \{2,3,\ldots,k\}$, we can find two  distinct walks of length $p$ between the same endpoints in the following walks in $D(A)$
$$\left\{\begin{array}{l}
 \cdots\rightarrow  k+s\rightarrow k+1\rightarrow k+2\rightarrow\cdots\rightarrow k+s\rightarrow k+1\rightarrow i\rightarrow m,\\
 \cdots\rightarrow  k+s\rightarrow k+1\rightarrow k+2\rightarrow\cdots\rightarrow k+s\rightarrow k+1\rightarrow j\rightarrow m.
\end{array}\right.$$
Therefore,  $A\notin \Gamma(k+s,p)$ for  any $p\in \{2,3,\ldots,k\}$.\\
\par

{\it (ii)}~~Without loss of generality, we assume $f(x_1)=f(y_1)=2$ and $a_{p_1,k+1}=a_{p_2,k+1}=a_{k+1,q_1}=a_{k+1,q_2}=1$ with $1\le p_1<p_2\le k$ and $1\le q_1<q_2\le k$.
If $p\ge 5$ is  odd, then $D(A)$ has the following two distinct walks of length $p$ from $p_1$ to $q_2$:
$$\left\{\begin{array}{l}
 p_1 \rightarrow  k+1\rightarrow k+2\rightarrow\cdots \rightarrow k+1  \rightarrow q_1 \rightarrow q_2,\\
  p_1\rightarrow p_2 \rightarrow  k+1\rightarrow k+2\rightarrow\cdots \rightarrow k+1  \rightarrow q_2.
\end{array}\right.$$
If $p\ge 5$ is  even, then $D(A)$ has the following two distinct walks of length $p$ from $p_1$ to $q_2$:
$$\left\{\begin{array}{l}
 p_1 \rightarrow p_2\rightarrow  k+1\rightarrow k+2\rightarrow\cdots \rightarrow k+1  \rightarrow q_1 \rightarrow q_2,\\
  p_1\rightarrow  k+1 \rightarrow  k+2\rightarrow k+1\rightarrow\cdots\rightarrow k+2 \rightarrow k+1  \rightarrow q_2.
\end{array}\right.$$

Therefore, $A\not\in \Gamma(k+s,p)$ for any integer $p\ge 5$.
\end{proof}

\begin{theorem}\label{th9}
 Let $n\ge 9$ be an integer. Then a digraph $D$ is in $Ex(n,\mathscr{F}_{n-4})$ if and only if $A_D$ is permutation similar to one of the following matrices
  $$
F_{1}(n)\equiv\begin{bmatrix}
             0&J_{3,n-5}&J_{3,2}\\
             0&T_{n-5}&J_{n-5,2}\\
             0&0&0\\
\end{bmatrix},\quad
F_{2}(n)\equiv\begin{bmatrix}
             0&J_{2,n-5}&J_{2,3}\\
             0&T_{n-5}&J_{n-5,3}\\
             0&0&0\\
\end{bmatrix},$$$$
F_{3}(n)\equiv\begin{bmatrix}
             0&J_{2,n-5}&J_{2,2}&J_{2,1}\\
             0&T_{n-5}&J_{n-5,2}&U_{m}\\
             0&0&0&0\\
             0&U'_{m}&J_{1,2}&0
\end{bmatrix},\quad
F_{4}(n)\equiv\begin{bmatrix}
T_{n-4}&w_{1}&w_{2}&w_{3}&w_{4}\\
u_{1}&0&1&1&1\\
u_{2}&0&0&1&1\\
u_{3}&0&0&0&1\\
u_{4}&0&0&0&0
\end{bmatrix},$$ where
$$U_{m}=(J_{1,m},0)^T,~~  U'_{m}=(0,J_{1,n-m-7}) \textrm{~~ with~~} 0\leq m\leq n-7$$  and
 $$w_{j}=(J_{1,k_j},0)^T, ~~u_{j}=(0,J_{1,n-k_j-5})  \textrm{~~for~~} i=1,2,3,4$$ with $0\leq k_1<k_2<k_3<k_4\leq n-5$.
 \end{theorem}
\begin{proof}
Suppose $A\equiv A_D$ is permutation similar to one of $F_1(n),F_2(n),F_3(n)$ and $F_4(n)$.   It is clear that $f(A)=\frac{n(n-1)}{2}-4$. To prove $D\in Ex(n,\mathscr{F}_{n-4})$, it suffices to prove $F_i^{n-4}(n)\in M_n\{0,1\}$   for $i=1,2,3,4$.

By direct computation we know
  $$F_{1}^{n-4}(n)=\begin{bmatrix}
             0&J_{3,2}\\
             0&0
\end{bmatrix},\quad F_{2}^{n-4}(n)=\begin{bmatrix}
              0&J_{2,3}\\
             0&0
\end{bmatrix},\quad  F_{3}^{n-4}(n)=\begin{bmatrix}
             0&J_{2,2}&O_{2\times 1}\\
             0&0&0
\end{bmatrix}$$ are all 0-1 matrices,
where $O_{2\times 1}$ is the $2\times 1$ zero matrix.

For $i=4$, let $\alpha=\{1,2,\ldots,n-4,n-3+k_1,n-3+k_2,n-3+k_3,n-3+k_4\}$ and
$$A' \equiv J_2\otimes T_{n-4}=\begin{bmatrix}
T_{n-4}&T_{n-4}\\
T_{n-4}&T_{n-4}
\end{bmatrix}.$$
Then  $F_4(n)=A'[\alpha]$ is a principal submatrix of $A'$. Moreover, $$(A')^{n-4}=(J_2\otimes T_{n-4})^{n-4}=J_2^{n-4}\otimes T_{n-4}^{n-4}=0$$ implies $F_4^{n-4}(n)=0$. Thus we get the sufficiency of Theorem \ref{th9}.

Next we prove the necessity part of Theorem \ref{th9}.
Suppose $D\in Ex(n,\mathscr{F}_{n-4})$.
Denote by $A\equiv A_D$ the adjacency matrix of $D$. Then $f(A)=ex(n,\mathscr{F}_{n-4})$, $A\in \Gamma(n,n-4)$ and  $A(i)\in \Gamma(n-1,n-4)$ for all $1\le i\le n$. By (\ref{eqh5}) we have
$$f(A(i))\le \frac{(n-1)(n-2)}{2}-2  \textrm{~~for all~~} 1\le i\le n.$$
We distinguish two cases.

{\it Case 1.} $f(A(q))=\frac{(n-1)(n-2)}{2}-2$ for some $q\in\{1,2,\ldots,n\}$. By Theorem \ref{leh6}, without loss of generality, we assume $q=n$ and
$$A=\begin{bmatrix}
             0&J_{2,n-5}&J_{2,2}&x_{1}\\
             0&T_{n-5}&J_{n-5,2}&x_{2}\\
             0&0&0&x_{3}\\
             y^{T}_{1}&y^{T}_{2}&y^{T}_{3}&\alpha
\end{bmatrix},$$
where $x_1,x_3,y_1, y_3\in \mathbb{R}^{2}$, $x_2,y_2\in \mathbb{R}^{n-5}$. Let $x=(x_1^T,x_2^T,x_3^T)$ and $y=(y_1^T,y_2^T,y_3^T)$. Since
\begin{equation}\label{eq319}
f(x)+f(y)+\alpha=f(A)-f(A(n))=n-3\ge 6,
\end{equation}
applying Lemma \ref{le4}, we have
\begin{equation}\label{eq318}
\alpha=0, ~y_{1}=x_{3}=0,\textrm{~~and~~ }a_{in}a_{nj}=0\textrm{~~for~~}j\le
i+2.
\end{equation}

If  $x=0$, then $A$ is permutation similar to $F_{1}(n)$. If  $y=0$, then $A=F_{2}(n)$.
If both $x$ and $y$ are nonzero, let $a_{sn}$ be the last nonzero component in $x$ and $a_{nt}$ be the first nonzero component in $y$. By (\ref{eq319}) and (\ref{eq318}) we have
\begin{equation}\label{eq320}
t-s=3,~~x=(J_{1,s},0)\textrm{~~and~~}y=(0,J_{1,n-s-3}).
\end{equation}
By exchanging row 1 and row 2 of $A$, and exchanging column 1 and column 2 of $A$ simultaneously, we obtain a new matrix $A'=(a'_{ij})$. Applying Lemma \ref{le4} to $A'$ we have $a'_{in}a'_{nj}=0$\textrm{ for }$j\le
i+2$.   Hence $$a_{1n}a_{n4}=a'_{2n}a'_{n4}=0.$$
Similarly, by interchanging the roles of the indices $n-1$ and $n-2$,  we get $$a_{n-4,n}a_{n,n-1}=0.$$ Therefore, in (\ref{eq320}) we have  $s\ne 1$ and $t\ne n-1$. Hence $A=F_3(n)$ with $m=s-2$.\\
\par

Case 2. $f(A(i))\le\frac{(n-1)(n-2)}{2}-3$ for all $i$. Denote by $$\delta_{i}=\sum_{j=1}^na_{ij}+\sum_{j\ne i}a_{ji}$$ the number of nonzero entries in the $i$-th row and the $i$-th column of $A$.
   Then \begin{equation}\label{eq321}
\delta_{i}= f(A)-f(A(i))\ge n-2 \textrm{~~for all~~} 1\le i\le n.
\end{equation}
Applying Lemma \ref{le3} to $A$, there exists some $i_0$ such that $$f(A(i_0))=\frac{(n-1)(n-2)}{2}-3.$$  Without loss of generality, we assume $i_0=n$.
For any $i\in\{1,\ldots,n-1\}$, since $A(i,n)\in \Gamma(n-2,n-4)$, by Theorem \ref{le6} we have
\begin{equation}\label{eq322}
f(A(i,n))\le\frac{(n-2)(n-3)}{2}-1.
\end{equation}

Next we prove the following claim.

{\it{\bf Claim 1.} Let $i\in\{1,2,\ldots,n-1\}$. Then
\begin{equation}\label{eq3323}
f(A(i,n))\le\frac{(n-2)(n-3)}{2}-2.
\end{equation}
Moreover, there exists some  $i\in\{1,2,\ldots,n-1\}$ such that  equality holds in (\ref{eq3323}). }

 Suppose equality in (\ref{eq322}) holds for some $i_1$, say, $i_1=n-1$. Then  by Theorem \ref{le6}, $A(n-1,n)$ is permutation similar to $K_{n-2}$ or $K'_{n-2}$.
  By (\ref{eq321}) we have
\begin{eqnarray*}
 \delta_{n-1}&=&f(A(n))-f(A(n-1,n))+a_{n-1,n}+a_{n,n-1}\\
 &=&n-4+a_{n-1,n}+a_{n,n-1}\\
 &\ge& n-2.
 \end{eqnarray*}
  It follows that $a_{n-1,n}=a_{n,n-1}=1$.

  If $A(n-1,n)$ is permutation similar to $K'_{n-2}$, without loss of generality, we may assume
$$A=\begin{bmatrix}
             0&J_{2,n-4}&x_{1}&x_{3}\\
             0&T_{n-4}&x_{2}&x_{4}\\
             y^{T}_{1}&y^{T}_{2}&\alpha&1\\
             y^{T}_{3}&y^{T}_{4}&1&\alpha'
\end{bmatrix},$$
where $x_1,x_3,y_1, y_3\in \mathbb{R}^{2}$, $x_2,y_2,x_4,y_4\in \mathbb{R}^{n-4}$. By  \cite[Lemma 1]{HZ1}, we have $\alpha=\alpha'=0$ and $$a_{in}+a_{ni}\le 1\textrm{~~for all~~}1\le i\le n-2.
$$
 Thus  $f(x_3)+f(y_3)\le 2$ and
 \begin{equation}\label{eq324}
 f(x_3)+f(x_4)+f(y_3)+f(y_4)=\delta_n-2=n-4\ge 5.
 \end{equation}

If $x_4$ has two nonzero entries, say, $a_{i_1,n}=a_{i_2,n}=1$ with $3\le i_1,i_2\le n-2$, then $D$ has the following distinct walks of length $n-4$  between the same endpoints
$$\left\{\begin{array}{l}
2\rightarrow i_1\rightarrow n\rightarrow n-1\rightarrow n\rightarrow\cdots\rightarrow n(\rightarrow n-1),\\
2\rightarrow i_2\rightarrow n\rightarrow n-1\rightarrow n\rightarrow\cdots\rightarrow n(\rightarrow n-1).
\end{array}\right.$$
Hence $f(x_4)\le 1$ and
 \begin{equation}\label{eq325}
 f(y_4)=n-4-f(x_3)-f(y_3)-f(x_4)\ge n-7\ge 2.
 \end{equation}

If $y_3$ and $y_4$ have three nonzero entries, say, $a_{n, j_1}=a_{n, j_2}=a_{n, j_3}=1$ with $j_1<j_2<j_3\le n-2$ and $j_2\ge 3$, then $D$ has the following distinct walks of length $n-4$ between the same endpoints
$$\left\{\begin{array}{l}
(n-1\rightarrow) n\rightarrow n-1\rightarrow n\rightarrow\cdots\rightarrow n\rightarrow n-1\rightarrow n\rightarrow j_3,\\
(n-1\rightarrow) n\rightarrow n-1\rightarrow n\rightarrow\cdots\rightarrow n \rightarrow j_1\rightarrow j_2\rightarrow j_3.
\end{array}\right.$$
Combining this with (\ref{eq324}) and (\ref{eq325}) , we have
 $$f(y_4)=2,~~y_3=0,~~f(x_3)=2 \textrm{~~and~~}f(x_4)=1.$$
Suppose the nonzero entries in $y_4$ are $a_{n,j_1}$ and $a_{n,j_2}$ with $3\le j_1<j_2\le n-2$.
Then  $D$ has two distinct walks of length $n-4$ between the same endpoints in the following walks:
$$\left\{\begin{array}{l}
1  \rightarrow 3\rightarrow 4\rightarrow\cdots\rightarrow n-2,\\
1  \rightarrow n \rightarrow n-1\rightarrow n\rightarrow\cdots\rightarrow n \rightarrow j_2 \rightarrow n-2 ,\\
1  \rightarrow n \rightarrow n-1\rightarrow n\rightarrow\cdots\rightarrow n \rightarrow j_1\rightarrow j_2 \rightarrow n-2 ,
\end{array}\right.$$
where $j_2\rightarrow n-2$ does not appear when $j_2=n-2$.
This  contradicts the condition that $A\in \Gamma(n,n-4)$.

 If $A(n-1,n)$ is permutation similar to $K_{n-2}$, then we may assume
$$A=\begin{bmatrix}
             T_{n-4}&J_{n-4, 2}&y_{2}&y_{4}\\
             0&0&y_{1}&y_{3}\\
             x^{T}_{2}&x^{T}_{1}&\alpha&1\\
             x^{T}_{4}&x^{T}_{3}&1&\alpha'
\end{bmatrix},$$
where $x_1,x_3,y_1, y_3\in \mathbb{R}^{2}$, $x_2,y_2,x_4,y_4\in \mathbb{R}^{n-4}$. Applying similar arguments as above we can deduce $A\not\in \Gamma(n,n-4)$, which contradicts $D\in Ex(n,\mathscr{F}_{n-4})$.
Hence (\ref{eq322}) is a strict inequality and we get (\ref{eq3323}).

 On the other hand, applying Lemma \ref{le3} to $A(n)$, there exists some  $i\in\{1,2,\ldots,n-1\}$ such that  equality in  (\ref{eq3323})
 holds. Thus we get Claim 1.\\

\par

 \vskip 8pt
Now without loss of generality we assume
  $f(A(n-1,n))=\frac{(n-2)(n-3)}{2}-2$.  Then $A(i,n-1,n)\in \Gamma(n-3,n-4)$ and
  \begin{equation}\label{eqq26}
  f(A(i,n-1,n))\le\frac{(n-3)(n-4)}{2}\textrm{~~for all~~} 1\le i\le n-2.
  \end{equation}
   Next we prove the following claim.\\
  \par

{\it{\bf Claim 2.}
 Let $i\in\{1,2,\ldots,n-2\}$. Then
 \begin{equation}\label{eq326}
 f(A(i,n-1,n))\le\frac{(n-3)(n-4)}{2}-1.
 \end{equation}
Moreover, there exists some  $i\in\{1,2,\ldots,n-2\}$ such that  equality in (\ref{eq326}) holds.}\\
 \par

 Suppose equality in (\ref{eqq26}) holds for some $i$, say, $i=n-2$. By Theorem \ref{thh1}, $A(n-2,n-1,n)$ is permutation similar to $T_{n-3}$. Without loss of generality, we assume
$$A=\begin{bmatrix}
             T_{n-3}&x_{1}&x_{2}&x_{3}\\
             y^{T}_{1}&a_{n-2,n-2}&a_{n-2,n-1}&a_{n-2,n}\\
             y^{T}_{2}&a_{n-1,n-2}&a_{n-1,n-1}&a_{n-1,n}\\
             y^{T}_{3}&a_{n,n-2}&a_{n,n-1}&a_{n,n}
\end{bmatrix},$$where $x_i,y_i\in \mathbb{R}^{n-3}$, for $i=1,2,3$.
 Since $\delta_{n-2}\ge n-2$  and $$f(x_{1})+f(y_{1})+a_{n-2,n-2}=f(A(n-1,n))-f(A(n-2,n-1,n))=n-5,$$
  we have
  $$\sum\limits_{i=n-1,n}(a_{n-2,i}+a_{i,n-2})=\delta_{n-2}-[f(x_{1})+f(y_{1})+a_{n-2,n-2}]\ge 3. $$
  Then either $a_{n-2,n-1}+a_{n-1,n-2}=2$ or $a_{n-2,n}+a_{n,n-2}=2$.
  Without loss of generality, we assume $a_{n-1,n-2}=a_{n-2,n-1}=1$. By Lemma 1 (ii) of \cite{HZ1}, we obtain $a_{n-1,n-1}=a_{n-2,n-2}=0$ and
  $$f(x_1)+f(y_1)=n-5\ge 4.$$
  Then we have $f(x_1)\ge 3$ or $f(y_1)\ge 3$, or $f(x_1)=f(y_1)=2$.
  Applying Lemma \ref{le10} to $A(n)$, we get $D\notin  Ex(n,\mathscr{F}_{n-4})$, a contradiction.

 Therefore, (\ref{eqq26}) is a strict inequality and we have  (\ref{eq326}).
Moreover, applying Lemma \ref{le3} to $A(n-1,n)$ we get the second part of Claim 2. \\

\par

Without loss of generality we assume $$f(A(n-2,n-1,n))=\frac{(n-3)(n-4)}{2}-1.$$
For any  $i\in \{1,2,\ldots,n-3\}$,  since $A(i,n-2,n-1,n)\in \Gamma(n-4,n-4)$,   by Theorem \ref{thh1}  we have
\begin{equation}\label{eq327}
f(A(i,n-2,n-1,n))\le\frac{(n-4)(n-5)}{2}.
\end{equation}
Applying  Lemma \ref{le3} to $A(n-2,n-1,n)$,  there is some $i$, say, $i=n-3$ such that equality in (\ref{eq327}) holds. It follows that $A(n-3,n-2,n-1,n)$ is permutation similar to $T_{n-4}$ and we may assume
 $$A=(a_{ij})=\begin{bmatrix}
T_{n-4}&x_{n-3}&x_{n-2}&x_{n-1}&x_{n}\\
y_{n-3}^T&a_{n-3,n-3}&a_{n-3,n-2}&a_{n-3,n-1}&a_{n-3,n}\\
y_{n-2}^T&a_{n-2,n-3}&a_{n-2,n-2}&a_{n-2,n-1}&a_{n-2,n}\\
y_{n-1}^T&a_{n-1,n-3}&a_{n-1,n-2}&a_{n-1,n-1}&a_{n-1,n}\\
y_{n}^T&a_{n,n-3}&a_{n,n-2}&a_{n,n-1}&a_{n,n}
\end{bmatrix}$$ where $x_{i},y_{i}\in \mathbb{R}^{n-4}$  for $i=n,n-1,n-2,n-3$.

Since
\begin{eqnarray*}
 f(x_{n-3})+f(y_{n-3})+a_{n-3,n-3}
 =f(A(n-2,n-1,n))-f(A(n-3,n-2,n-1,n))
 =n-5,
\end{eqnarray*}
   applying Lemma \ref{le7}  to $A(n-2,n-1,n)$ we have $a_{n-3,n-3}=0$ and  \begin{eqnarray}\label{eqn328}
 f(x_{n-3})+f(y_{n-3})
 =n-5\ge 4.
\end{eqnarray}

   We assert
    \begin{equation}\label{eq328}
     a_{n-3,n-2}a_{n-2,n-3}=0.
    \end{equation}
  Otherwise  $a_{n-3,n-2}=a_{n-2,n-3}=1$. Applying Lemma \ref{le10} to $A(n-1,n)$ we can deduce $A\not\in \Gamma(n,n-4)$, a contradiction.

  Similarly, we have
  $$a_{n-3,i}a_{i,n-3}=0 \textrm{~~for~~} i=n-1,n.$$
  It follows that $$\sum\limits_{i=n-2}^{n}(a_{n-3,i}+a_{i,n-3})\le 3.$$ On the other hand, $$\sum\limits_{i=n-2}^{n}(a_{n-3,i}+a_{i,n-3})=\delta_{n-3}-f(x_{n-3})-f(y_{n-3})\ge n-2-(n-5)=3.$$
  Hence, we have $\sum\limits_{i=n-2}^{n}(a_{n-3,i}+a_{i,n-3})= 3$ and $$a_{n-3,i}+a_{i,n-3}=1\textrm{~~for~~} i=n-2,n-1,n.$$   Applying Lemma \ref{le7} to $A(n-3,n-1,n)$ we obtain $a_{n-2,n-2}=0$ and
   \begin{eqnarray*}
   &&  f(x_{n-2})+f(y_{n-2})\\
   && =f(A(n-1,n))-f(A(n-2,n-1,n))-a_{n-2,n-2}-a_{n-3,n-2}-a_{n-2,n-3}\\
   &&=n-5.
   \end{eqnarray*}

     Repeating the  above arguments, we get $$a_{n-2,i}+a_{i,n-2}=1\textrm{~~for~~} i=n-1,n$$
      and $$a_{n-1,n-1}=a_{nn}=0,~~a_{n-1,n}+a_{n,n-1}=1.$$
       Moreover, we have
   \begin{equation}\label{eq330}
   f(x_{n-i})+f(y_{n-i})=n-5\textrm{~~for~~}i=0,1,2,3.
   \end{equation}

Now we verify\\
\par

{ \it {\bf Claim 3.} $B\equiv A[n-3,n-2,n-1,n]$ is permutation similar to $T_{4}$.}\\
\par

It is well known that the adjacency matrix of an acyclic digraph is permutation similar to a strictly upper triangular matrix. Suppose Claim 3 does not hold.  Then the digraph $D(B)$  has at least one cycle. Note that $a_{ii}=0$ and $a_{ij}a_{ji}=0$ for $i,j=n-3,\ldots,n$. $D(B)$ has no loop or 2-cycle.  If $D(B)$ has a 4-cycle, then $B$ is permutation similar to
$$B'=\begin{bmatrix}
0&1&b_{13}&0\\
0&0&1&b_{24}\\
b_{31}&0&0&1\\
1&b_{42}&0&0
\end{bmatrix}.$$
Now $b_{24}=1$ implies  $D(B)$ has a cycle $1\rightarrow 2\rightarrow 4\rightarrow 1$ of length $3$;   $b_{42}=1$ implies  $D(B)$ has a cycle $2\rightarrow 3\rightarrow 4\rightarrow 2$ of length $3$. Therefore, $D$ always  has a 3-cycle with vertices from $\{n-3,n-2,n-1,n\}.$
Without loss of generality, we assume the 3-cycle is  $n-3\rightarrow n-2\rightarrow n-1\rightarrow n-3$.

 Applying Lemma \ref{le10} to $A(n)$, we get
$$f(x_{i})\le 2 \textrm{~~and~~} f(y_{i})\le 2 \textrm{~~ for~~} i=n-3,n-2,n-1.$$ By (\ref{eq330}) we have $n=9$ and
$$f(x_{i})=f(y_{i})=2 \textrm{~~for~~} i=n-3,n-2,n-1.$$
Suppose $a_{j_1,6}=a_{j_2,6}=a_{8,j_3}=a_{8,j_4}=1$ with $1\le j_1<j_2\le n-4=5$ and $1\le j_3<j_4\le 5$. Then $D$ has   two distinct walks from $j_1$ to $j_4$ of length $n-4$:
$$\left\{\begin{array}{l}
 j_1\rightarrow 6\rightarrow 7\rightarrow 8\rightarrow j_{3}\rightarrow j_4,\\
 j_1\rightarrow j_2\rightarrow 6\rightarrow 7\rightarrow 8\rightarrow j_4,
\end{array}\right.$$
a contradiction.
Therefore, $D(B)$ is acyclic  and Claim 3 holds. \\
\par

Without loss of generality, we assume
$$A=\begin{bmatrix}
T_{n-4}&x_{n-3}&x_{n-2}&x_{n-1}&x_{n}\\
y_{n-3}^T&0&1&1&1\\
y_{n-2}^T&0&0&1&1\\
y_{n-1}^T&0&0&0&1\\
y_{n}^T&0&0&0&0
\end{bmatrix}.$$

For $i=n,n-1,n-2,n-3$, let $a_{s_{i},i}$ be the last nonzero component in $x_{i}$ and $a_{i,t_{i}}$ be the first nonzero component in $y_{i}$, where $s_{i}\equiv 0$ if $x_{i}=0$ and $t_{i}\equiv n-3$  if $y_i=0$.  Applying Lemma \ref{le7} to $A(n-2,n-1,n)$, $A(n-3,n-1,n)$, $A(n-3,n-2,n)$  and $A(n-3,n-2,n-1)$,  we have
\begin{equation}\label{eqn331}
x_i=(a_i^T,0)^T,~~y_i=(0,b_i^T)^T \textrm{~~for~~} i=n-3,n-2,n-1,n,
\end{equation}
 where $a_i\in \mathbb{R}^{s_i}$, $b_i\in \mathbb{R}^{n-s_i-4}$. Moreover, by (\ref{eq330}) we have
 \begin{equation}\label{eq331}
s_{i}<t_{i}\le s_{i}+2\textrm{~~for~~}i=n,n-1,n-2,n-3.
\end{equation}

Next, we verify the following claim.\\
\par

{\it{\bf Claim 4.}  If $n-3\le i<j\le n$, then
\begin{equation}\label{eq332}
t_{j}>s_{i}+1.
\end{equation}}
\par

We assert
\begin{equation}\label{eq34}
a_{ij}=0 \textrm{~~for~~}n-5\le i\le n-4, n-3\le j\le i+3.
 \end{equation}
 Otherwise, $D$ has the following two distinct walks of length $n-4$ from $1$ to $j+1$ or $j+2$:
$$\left\{\begin{array}{l}
 1\rightarrow 2\rightarrow 4\rightarrow\cdots\rightarrow i\rightarrow j\rightarrow j+1 (\rightarrow j+2), \\
 1\rightarrow 3\rightarrow 4\rightarrow\cdots \rightarrow i\rightarrow j\rightarrow j+1 (\rightarrow j+2).
\end{array}\right.$$
 Similarly, we have $a_{i1}=0$ for $i=n-2,n-1,n$. Otherwise, $D$ has the following two distinct walks of length $n-4$ from $i-1$ to $n-4$:
$$\left\{\begin{array}{l}
 i-1\rightarrow i\rightarrow 1\rightarrow 2\rightarrow 4\rightarrow\cdots\rightarrow n-4,\\
 i-1\rightarrow i\rightarrow 1\rightarrow 3\rightarrow 4\rightarrow\cdots\rightarrow n-4.
\end{array}\right.$$

Given any $n-3\le i<j\le n$,
we have $$t_j\ge 2 \textrm{~~for~~} n-2\le j\le n$$ and $$s_{i}\le n-5 \textrm{~~for~~} n-3\le i\le n-1.$$ If $s_{i}=0$, then  $t_{j}>s_{i}+1$ and (\ref{eq332}) holds.  For $s_{i}\ge 1$, if (\ref{eq332}) does not hold, we can distinguish the following cases to find two distinct walks of length $n-4$ between the same endpoints to deduce contradictions.

{\it Subcase 1.}  $t_{j}= s_{i}+1$.
 If $s_{i}=1$, then $t_j=2$ and $D$ has
$$\left\{\begin{array}{l}
 1\rightarrow i\rightarrow j\rightarrow t_{j}\rightarrow 4\rightarrow 5\rightarrow \cdots\rightarrow n-4,\\
 1\rightarrow i\rightarrow j\rightarrow t_{j}\rightarrow 3\rightarrow 5\rightarrow \cdots\rightarrow n-4.
\end{array}\right.$$
If $s_{i}=2$ and $a_{1i}=0$, then by (\ref{eq330}) and (\ref{eqn331}), $t_j=3$, $t_{i}=s_{i}+1=3$ and $D$ has
$$\left\{\begin{array}{l}
 1\rightarrow 2\rightarrow i\rightarrow j\rightarrow t_{j}\rightarrow 5\rightarrow 6\rightarrow \cdots\rightarrow n-4,\\
 1\rightarrow 2\rightarrow i\rightarrow t_{i}\rightarrow 4\rightarrow 5\rightarrow \cdots\rightarrow n-4.
\end{array}\right.$$
If $s_{i}=2$ and $a_{1i}=1$, then   $D$ has
$$\left\{\begin{array}{l}
 1\rightarrow 2\rightarrow i\rightarrow j\rightarrow t_{j}\rightarrow 5\rightarrow 6\rightarrow \cdots\rightarrow n-4,\\
 1\rightarrow i\rightarrow j\rightarrow t_{j}\rightarrow 4\rightarrow 5\rightarrow 6\rightarrow \cdots\rightarrow n-4.
\end{array}\right.$$
If $s_{i}=3$ and $a_{2i}=0$, then $t_{i}=s_{i}+1=4$ and $D$ has
$$\left\{\begin{array}{l}
 1\rightarrow 3\rightarrow i\rightarrow j\rightarrow t_{j}\rightarrow 5\rightarrow 6\rightarrow \cdots\rightarrow n-4,\\
 1\rightarrow 2\rightarrow 3\rightarrow i\rightarrow t_{i}\rightarrow 5\rightarrow 6\rightarrow \cdots\rightarrow n-4.
\end{array}\right.$$
If $s_{i}=3$ and $a_{2i}=1$, then $D$ has
$$\left\{\begin{array}{l}
 1\rightarrow 3\rightarrow i\rightarrow j\rightarrow t_{j}\rightarrow 5\rightarrow 6\rightarrow \cdots\rightarrow n-4,\\
 1\rightarrow 2\rightarrow i\rightarrow j\rightarrow t_{j}\rightarrow 5\rightarrow 6\rightarrow \cdots\rightarrow n-4.
\end{array}\right.$$
If  $ 4\le s_{i}\le n-5$, then $D$ has
$$\left\{\begin{array}{l}
 1\rightarrow 3\rightarrow 4\rightarrow \cdots \rightarrow s_{i}\rightarrow i\rightarrow j\rightarrow t_{j} \rightarrow  t_{j}+1\rightarrow \cdots\rightarrow n-4 ,\\
 1\rightarrow 2\rightarrow 4\rightarrow \cdots \rightarrow s_{i}\rightarrow i\rightarrow j\rightarrow t_{j} \rightarrow  t_{j}+1\rightarrow \cdots\rightarrow n-4 ,
\end{array}\right.$$
 where the walk $t_j\rightarrow  t_{j}+1\rightarrow \cdots\rightarrow n-4$ does not appear when $s_i=n-5$.

{\it Subcase 2.} $t_{j}<s_{i}+1$. Since $t_{j}\ge 2$ for $j=n-2,n-1,n$, we have $s_{i}\ge 2$. If $s_i=2$ or $3$, $D$ has the same walks as in the previous subcase.
If $4\le s_{i}\le n-5$, the walks
$$\left\{\begin{array}{l}
 1\rightarrow 3\rightarrow 4\rightarrow \cdots\rightarrow s_{i}\rightarrow i\rightarrow j\rightarrow t_{j}\rightarrow  s_{i}+1\rightarrow \cdots\rightarrow n-4   \\
 1\rightarrow 2\rightarrow 4\rightarrow \cdots\rightarrow s_{i}\rightarrow i\rightarrow j\rightarrow t_{j}\rightarrow  s_{i}+1\rightarrow \cdots\rightarrow n-4
\end{array}\right.$$
 contain two walks of length $n-4$ with the same endpoints.

 Thus we obtain Claim 4.\\
 \par

Now we show

{\it {\bf Claim 5.}
$s_i\ne s_{j}$\textrm{ for }$n-3\le i<j\le n$.}\\
\par

 Suppose $s_k=s_{l}$ for some $n-3\le k<l\le n$. Then by the definition of $s_{i}$ we have
 \begin{equation}\label{eq335}
 a_{s_{k}+1,k}=a_{s_{k}+1,l}=0.
 \end{equation}
 From (\ref{eqn331}) we have
  \begin{equation}\label{eq336}
  a_{s_{k}+1,j}+a_{j,s_{k}+1}\le 1\textrm{~~for~~ }j=n-3,n-2,n-1,n.
  \end{equation}
By (\ref{eq321}),
 $$\delta_{s_{k}+1}=n-5+\sum\limits_{j=n-3}^{n}(a_{s_{k}+1,j}+a_{j,s_{k}+1})\ge n-2.$$
Combining this with (\ref{eq335}) and (\ref{eq336}) we have
 $$a_{k,s_{k}+1}+a_{l,s_{k}+1}\ge1.$$
    By (\ref{eq331}) and (\ref{eq332}), we have $t_{l}=s_{l}+2$. Hence, $$a_{l,s_{k}+1}=0\textrm{~~and~~ } a_{k,s_{k}+1}=1.$$
     It follows that
     \begin{equation}\label{eq337}
     t_{k}=s_{k}+1.
     \end{equation}
  If $s_{k}=0$, then $D$ has two distinct walks of length $n-4$ from $k$ to $n-4$:
$$\left\{\begin{array}{l}
 k\rightarrow 1\rightarrow 2\rightarrow \cdots \rightarrow  n-4,\\
 k\rightarrow l\rightarrow 2\rightarrow \cdots \rightarrow n-4,
\end{array}\right.$$
a contradiction.
 If $1\le s_{k}\le n-6$,  then $D$ has two distinct walks of length $n-4$ from $1$ to $n-4$:
$$\left\{\begin{array}{l}
 1\rightarrow 2\rightarrow \cdots \rightarrow s_{k}\rightarrow k\rightarrow l\rightarrow t_{l}\rightarrow t_{l}+1\rightarrow\cdots\rightarrow n-4,\\
 1\rightarrow 2\rightarrow \cdots \rightarrow s_{k}\rightarrow k\rightarrow t_{k}\rightarrow t_{k}+1\rightarrow\cdots\rightarrow n-4,
\end{array}\right.$$
a contradiction. Hence we have $s_k\ge n-5$.

On the other hand, by (\ref{eq34}) we have
\begin{equation*}
s_{n-3}\le n-6,  s_{n-2}\le n-6\textrm{ and }s_{n-1}\le n-5.
 \end{equation*}
 It follows that $$s_{n-3}\ne s_{n-2},~~k=n-1,~~l=n\textrm{ ~~and~~} s_k=n-5.$$
Moreover, there exists  $s_m\ge 1$ with $m\in \{n-3,n-2\}$.
  Now we can distinguish the following cases to find distinct walks of length $n-4$ with the same endpoints in $D$ to deduce contradictions.

If $t_{m}=s_{m}+1$, then by (\ref{eq337}) $D$ has
$$\left\{\begin{array}{l}
 1\rightarrow 2\rightarrow \cdots \rightarrow s_{m}\rightarrow m\rightarrow s_{m}+1\rightarrow\cdots\rightarrow n-4,\\
 1\rightarrow 2\rightarrow \cdots \rightarrow s_{k}\rightarrow k\rightarrow n-4.
\end{array}\right.$$
If $s_{m}+2=t_{m}\le n-5$, then $D$ has
$$\left\{\begin{array}{l}
 1\rightarrow 2\rightarrow \cdots \rightarrow s_{m}\rightarrow m\rightarrow t_{m}\rightarrow \cdots\rightarrow s_k\rightarrow k\rightarrow n-4,\\
 1\rightarrow 2\rightarrow \cdots \rightarrow s_{k}\rightarrow k\rightarrow n-4.
\end{array}\right.$$
If $s_{m}+2=t_{m}= n-4$,  then $s_{m}=n-6$ and $D$ has
$$\left\{\begin{array}{l}
 1\rightarrow 2\rightarrow \cdots \rightarrow s_m\rightarrow m\rightarrow k\rightarrow n-4,\\
 1\rightarrow 2\rightarrow \cdots \rightarrow s_{k}\rightarrow k\rightarrow n-4.
\end{array}\right.$$
Hence we get Claim 5.\\

\par

  Combining Claim 4 and Claim 5  we obtain
  $$s_{j}>s_{i} \textrm{~~for~~} n-3\le i<j\le n.$$
   Otherwise,  $s_{j}<s_{i}$ leads to $t_{j}\le s_{j}+2< s_{i}+2$, which contradicts  Claim 4. Therefore, we have
\begin{equation}\label{eq338}
0\le s_{n-3}<s_{n-2}<s_{n-1}<s_{n}\le n-4.
\end{equation}

Finally, we verify

  {\it {\bf Claim 6.} $t_{i}=s_{i}+2$ for $i\in\{n-3,n-2,n-1,n\}$.}\\

 \par
Suppose $t_{n}=s_{n}+1$. By (\ref{eq338}) and (\ref{eq331}) we have $1\le s_{n-2}\le n-6$ and $s_{n}\ge t_{n-2}$. We can distinguish the following cases to find  two distinct walks of length $n-4$ from $1$ to $n-4$ or $n$ in $D$, which contradicts $D\in Ex(n,\mathscr{F}_{n-4})$.
If $t_{n-2}=s_{n-2}+2$,  then $D$ has
$$\left\{\begin{array}{l}
 1\rightarrow 2\rightarrow \cdots\rightarrow s_{n}\rightarrow n \rightarrow s_{n}+1
 \rightarrow s_{n}+2\rightarrow \cdots\rightarrow n-4 , \\
 1\rightarrow 2\rightarrow \cdots\rightarrow s_{n-2}\rightarrow n-2\rightarrow s_{n-2}+2\rightarrow \cdots
 \rightarrow s_{n}\rightarrow n \rightarrow s_{n}+1\rightarrow \cdots\rightarrow n-4 ,
\end{array}\right.$$ where the walk $n \rightarrow s_{n}+1\rightarrow \cdots\rightarrow n-4$ does not appear when $s_n=n-4$.
If $t_{n-2}=s_{n-2}+1$ and $s_n\le n-5$,
then $D$ has
$$\left\{\begin{array}{l}
 1\rightarrow 2\rightarrow \cdots\rightarrow s_{n}\rightarrow n\rightarrow s_{n}+1
 \rightarrow s_{n}+2\rightarrow \cdots\rightarrow n-4,\\
 1\rightarrow 2\rightarrow \cdots\rightarrow s_{n-2}\rightarrow n-2\rightarrow s_{n-2}+1
 \rightarrow s_{n-2}+2\rightarrow \cdots\rightarrow n-4.
\end{array}\right.$$
If $t_{n-2}=s_{n-2}+1$ and $s_n= n-4$, then $a_{n-4,n}=1 $ and $D$ has
$$\left\{\begin{array}{l}
 1\rightarrow 2\rightarrow \cdots\rightarrow n-5\rightarrow n-4\rightarrow n,\\
 1\rightarrow 2\rightarrow \cdots\rightarrow s_{n-2}\rightarrow n-2\rightarrow t_{n-2}\rightarrow t_{n-2}+2\rightarrow \cdots\rightarrow n-4\rightarrow n, \textrm{ if }s_{n-2}\le n-7,\\
    1\rightarrow 3\rightarrow \cdots\rightarrow s_{n-2}\rightarrow n-2\rightarrow t_{n-2}\rightarrow t_{n-2}+1\rightarrow \cdots\rightarrow n-4\rightarrow n, \textrm{ if }s_{n-2}= n-6.
\end{array}\right.$$
 Hence, $t_{n}=s_{n}+2$ and $s_n\le n-5$.

Next suppose  $t_{i}=s_{i}+1$   for $i\in\{n-2,n-1\}$ and $j\in \{n-2,n-1\}\setminus\{i\}$. Then $s_i,s_j>0$.
 If $t_j=s_j+1$, then $D$ has the following two distinct walks of length $n-4$ from $1$ to $n-4$:
 $$\left\{\begin{array}{l}
 1\rightarrow 2 \rightarrow \cdots \rightarrow   s_{i}\rightarrow i\rightarrow s_{i}+1\rightarrow \cdots\rightarrow n-4,\\
 1\rightarrow 2 \rightarrow \cdots \rightarrow    s_{j}\rightarrow j\rightarrow s_{j}+1\rightarrow \cdots\rightarrow n-4,
\end{array}\right.$$
a contradiction. If $t_j=s_j+2$, then $D$ has the following two distinct walks of length $n-4$ from $1$ to $n-4$:
$$\left\{\begin{array}{l}
 1 \rightarrow \cdots \rightarrow  s_{i}\rightarrow i\rightarrow s_{i}+1\rightarrow \cdots\rightarrow n-4,\\
 1  \rightarrow \cdots \rightarrow s_{i}\rightarrow i\rightarrow s_{i}+1\rightarrow \cdots\rightarrow  s_{j}\rightarrow j\rightarrow s_{j}+2\rightarrow \cdots\rightarrow n-4,~~\textrm{if}~~i<j,\\
  1  \rightarrow \cdots \rightarrow s_{j}\rightarrow j\rightarrow s_{j}+2\rightarrow \cdots\rightarrow  s_{i}\rightarrow i\rightarrow s_{i}+1\rightarrow \cdots\rightarrow n-4,~~\textrm{if}~~i>j \textrm{ and } s_i\ge s_j+2\\
   1  \rightarrow \cdots \rightarrow s_{j}\rightarrow j\rightarrow i\rightarrow s_{i}+1\rightarrow  \cdots\rightarrow  n-4,~~\textrm{if}~~i>j \textrm{ and } s_i=s_j+1,
\end{array}\right.$$
  a contradiction.  Hence, we get
  $$t_{n-2}=s_{n-2}+2,\quad t_{n-1}=s_{n-1}+2.$$

  Now we  conclude $t_{n-3}=s_{n-3}+2$.  Otherwise $D$ has the following two distinct walks of length $n-4$ from $1$ or $n-3$ to $n-4$:
  $$\left\{\begin{array}{l}
  1\rightarrow 2\rightarrow   \cdots \rightarrow s_{n-3}\rightarrow  n-3\rightarrow t_{n-3}\rightarrow \cdots\rightarrow n-4,\\
  1\rightarrow 2\rightarrow   \cdots \rightarrow s_{n-3}\rightarrow  n-3\rightarrow t_{n-3}\rightarrow \cdots\rightarrow s_{n-1}\rightarrow n-1\rightarrow s_{n-1}+2\rightarrow \cdots\rightarrow  n-4,
\end{array}\right.$$
 where the walk $1\rightarrow 2 \rightarrow \cdots \rightarrow  s_{n-3} \rightarrow n-3$ does not appear when $s_{n-3}=0$.
Thus we get Claim 6.\\
\par

Finally, combining (\ref{eq330}) and Claim 6 we have $A=F_4(n)$. This completes the proof.

 \end{proof}

From the proof of Theorem \ref{th9}, we have the following corollary.
\begin{corollary}\label{co1}
 Let $k\ge 5$ be an integer, $n=k+4$ and
   $$A=(a_{ij})=\begin{bmatrix}
T_{k}&B\\
C&E
\end{bmatrix}\in \Gamma(n,k).$$
 If
\begin{eqnarray*}
 f(A)=\frac{n(n-1)}{2}-4, &&f(A(n))=\frac{(n-1)(n-2)}{2}-3,\\
   f(A(n-1,n))=\frac{(n-2)(n-3)}{2}-2,  &&f(A(n-2,n-1,n))=\frac{(n-3)(n-4)}{2}-1,
 \end{eqnarray*}
 then $A$ is permutation similar to $F_4(n)$  by permuting its last 4 rows and columns.
 \end{corollary}

\section{Proof of Theorem \ref{thh2}}
In this section we give the proof of Theorem \ref{thh2}. We will use induction on $n$. First we need the following lemma to show that Theorem \ref{thh2} holds for $k=5$.

\begin{lemma}\label{le14}
 Let $n\ge 10$ be an integer. Then
\begin{equation}\label{eq339}
ex(n,\mathscr{F}_{n-5})=\frac{n(n-1)}{2}-5.
 \end{equation}
 Moreover, a digraph $D$ is in $Ex(n,\mathscr{F}_{n-5})$  if and only if $D$ is a $(1,n-5,5)$-completely transitive tournament.
 \end{lemma}
\begin{proof}
Let $D$  be any $ \mathscr{F}_{n-5}$-free digraph of order $n$.
Denote by $A\equiv A_D$.
Given any $i\in \{1,2,\ldots,n\}$, since $A(i)\in \Gamma(n-1,n-5)$, by Theorem \ref{le8}
we have
\begin{equation}\label{eqq41}
f(A(i))\le \frac{(n-1)(n-2)}{2}-4.
\end{equation}
Applying Lemma \ref{le3} to $A$, we have
$$f(A)\le \frac{n(n-1)}{2}-5.$$
Hence,
 $$ex(n,\mathscr{F}_{n-5})\le \frac{n(n-1)}{2}-5.$$

Let $D$ be any $(1,n-5,5)$-completely transitive tournament. Then $A_D$ is a principal submatrix of
 $A'=
J_2\otimes T_{n-5}
 $. Since $A'(\alpha)\in\Gamma(n,n-5)$, the digraph $D(A')$, and hence $D$ is $\mathscr{F}_{n-5}$-free. It is clear that
 $D$ has size  $\frac{n(n-1)}{2}-5$. Thus we get (\ref{eq339}) and the sufficiency of the second part.
\\
\par

Let $D\in Ex(n,\mathscr{F}_{n-5})$ and  $A\equiv A_D$.
Again, denote by $\delta_i$ the number of nonzero entries lying in the $i$-th row and the $i$-th column of $A$. Then by (\ref{eqq41}),
\begin{equation}\label{eq40}
\delta_i=f(A)-f(A(i))\ge n-2 \textrm{~~for all~~}1\le i\le n.
\end{equation}

Applying Lemma \ref{le3} we get
$$f(A(i_0))\ge
\frac{(n-1)(n-2)}{2}-4$$ for some $i_0\in \{1,2,\ldots,n\}$.
Without loss of generality, we assume $i_0=n$. By Theorem \ref{th9}, $A(n)$ is permutation similar to   $F_t(n-1)$ with $t\in\{1,2,3,4\}$.
 Now we distinguish four cases.

{\it Case 1.} $A(n)$ is permutation similar to $F_1(n-1)$. Without loss of generality, we assume
$$A=\begin{bmatrix}
             0&J_{3,n-6}&J_{3,2}&x_{1}\\
             0&T_{n-6}&J_{n-6,2}&x_{2}\\
             0&0&0&x_{3}\\
             y_{1}^{T}&y_{2}^{T}&y_{3}^{T}&\alpha
\end{bmatrix},$$ where $x_{1},y_{1}\in \mathbb{R}^{3}$, $x_{2},y_{2}\in \mathbb{R}^{n-6}$, $x_{3}, y_{3}\in \mathbb{R}^{2}$.

 Since $$\delta_{n}=\sum\limits_{i=1}^{3}[f(x_{i})+f(y_{i})]+\alpha
=f(A)-f(A(n))=n-2,$$   applying Lemma \ref{le4} to $A$ we have
$$y_{1}=0\textrm{ and }x_{3}=0.$$
Then $\delta_1\le  n-3$, which contradicts (\ref{eq40}).

 {\it Case 2.} $A(n)$ is  permutation similar to $F_2(n-1)$.
  Applying the same arguments as in Case 1 we get $\delta_{n-2}\le n-3$, a contradiction.

  {\it Case 3.} $A(n)$ is  permutation similar to $F_3(n-1)$. Without loss of generality, we assume
  $$A=\begin{bmatrix}
             0&J_{2,n-6}&J_{2,2}&J_{2,1}&x_1\\
             0&T_{n-6}&J_{n-6,2}&U_{m}&x_2\\
             0&0&0&0&x_3\\
             0&U'_{m}&J_{1,2}&0&a_{n-1,n}\\
             y_1^T&y_2^T&y_3^T&a_{n,n-1}&a_{nn}
\end{bmatrix}$$
where $x_1,x_3,y_1,y_3\in \mathbb{R}^2$, $x_2,y_2\in \mathbb{R}^{n-6}$, $U_{m}=(J_{1,m},0)^T$, $U'_{m}=(0,J_{1,n-m-8})$, $0\le m\le n-8$.
By (\ref{eq40}), $$\delta_{n-1}=n-4+a_{n-1,n}+a_{n,n-1}\ge n-2$$
implies $a_{n-1,n}=a_{n,n-1}=1$.
Applying  Lemma \ref{le4} to $A(n-1)$ we get $x_3=0$ and $y_1=0$. Hence, $\delta_1\ge n-2$ and $\delta_2\ge n-2$ force  $x_1=J_{2,1}$. Then $D$ has the following two distinct walks of length $n-5$ from 1 to $n-1$ or $n$:
  $$\left\{\begin{array}{l}
 1\rightarrow 2\rightarrow n-1   \rightarrow n \rightarrow n-1\rightarrow\cdots \rightarrow n-1(\rightarrow n),\\
 1\rightarrow n\rightarrow n-1   \rightarrow n \rightarrow n-1\rightarrow\cdots \rightarrow n-1(\rightarrow n),
\end{array}\right.$$
a contradiction.

  {\it Case 4.} $A(n)$ is  permutation similar to $F_4(n-1)$. Without loss of generality we assume
$$A=(a_{ij})=\begin{bmatrix}
T_{n-5}&w_{4}&w_{3}&w_{2}&w_{1}&x\\
u_{4}&0&1&1&1&a_{n-4,n}\\
u_{3}&0&0&1&1&a_{n-3,n}\\
u_{2}&0&0&0&1&a_{n-2,n}\\
u_{1}&0&0&0&0&a_{n-1,n}\\
y^{T}&a_{n,n-4}&a_{n,n-3}&a_{n,n-2}&a_{n,n-1}&a_{n,n}
\end{bmatrix}\equiv \begin{bmatrix} T_{n-5}&B\\C&E\end{bmatrix},$$ where $x,y\in \mathbb{R}^{n-5}$, $$w_{i}=(J_{1,q_i},0)^T,~~  u_{i}=(0,J_{1,n-q_i-6}) \textrm{~~for~~} i\in \{1,2,3,4\} $$  with $0\leq q_4<q_3<q_2<q_1\leq n-6$.

We claim
\begin{equation}\label{eq41}
a_{n-i,n}a_{n,n-i}=0 \textrm{~~for~~}i=1,2,3,4.
\end{equation}
Otherwise suppose
  $a_{n-i,n}=a_{n,n-i}=1$ for some $i\in \{1,2,3,4\}$.
   Set $$\alpha=\{1,2,\ldots,n-5,n-i,n\}.$$
   Applying Lemma \ref{le10} to $A[\alpha]$ we get
     $A\not\in \Gamma(n,n-5)$, a contradiction. Thus we obtain (\ref{eq41}).

On the other hand,   by (\ref{eq40}) we have
$$a_{n,n-i}+a_{n-i,n}=\delta_{n-i}-f(w_i)-f(u_i)-3\ge 1\textrm{~~for~~} i\in \{1,2,3,4\}.$$
 Hence, we have $a_{n,n-i}+a_{n-i,n}=1$, and
 $$f(A(n-i))=f(A)-\delta_{n-i}=\frac{(n-1)(n-2)}{2}-4\textrm{~~for~~}i=1,2,3,4.$$

Given $i\in\{1,2,3,4,5\}$, applying Corollary \ref{co1} to $A(n-i)$ we know  each $4\times 4$ principal submatrix of $E$  is permutation similar to $T_4$. Let $w_5=x$ and $u_5=y^T$. By Lemma 9 of \cite{HZ1}, $E$ is permutation similar to $T_5$ and $A$ is permutation similar to
$$H=\begin{bmatrix}
T_{n-5}&w_{\sigma_1}&w_{\sigma_2}&w_{\sigma_3}&w_{\sigma_4}&w_{\sigma_5}\\
u_{\sigma_1}&0&1&1&1&1\\
u_{\sigma_2}&0&0&1&1&1\\
u_{\sigma_3}&0&0&0&1&1\\
u_{\sigma_4}&0&0&0&0&1\\
u_{\sigma_5}&0&0&0&0&0
\end{bmatrix}$$
with $\sigma$ a permutation of $\{1,2,3,4,5\}$. Applying Corollary \ref{co1} to each $A(n-i)$ again we get
  $$w_{\sigma_i}=(J_{1,k_i},0)^T,~~ u_{\sigma_i}=(0,J_{1,n-k_i-6})\textrm{~~for~~} i=1,2,3,4,5$$  with $0\le k_{1}<k_{2}<k_{3}<k_{4}<k_{5}\le n-6$.

Denote $G=J_2\otimes T_{n-5}$ and $$\beta=\{n-4,n-5,\ldots,2(n-5)\}
\setminus\{n-4+k_1,n-4+k_2,n-4+k_3,n-4+k_4,n-4+k_5\}.$$  Then $H=G(\beta)$ and $D$ is a $(1,n-5,5)$-completely transitive tournament.
   \end{proof}
 Now we are ready to present the proof of Theorem \ref{thh2}.\\
 \par

 {\bf Proof of Theorem \ref{thh2}.}~~
 We use induction on $n$. By  Theorem \ref{le8} and Lemma \ref{le14}  we  know the results hold  for $n=k+4$ and $n=k+5$. Assume  the results hold  for  $n=k+5,\ldots,sk+t$, where $0\le t<k$ and $s>0$ are integers.

 Now suppose $n=sk+t+1.$ Let $u,v$ be  integers  such that $v<k$ and $ n=uk+v$. Then $u=s$, $v=t+1$ when $t<k-1$, and $u=s+1$, $v=0$ when $t=k-1$.

Given any $ \mathscr{F}_{k}$-free digraph  $D$ of order $n$, denote by $A\equiv A_D$ its adjacency matrix.
For any  $ i\in\{1,2,\ldots,n\}$, since the digraph of $A(i)$ is an $\mathscr{F}_{k}$-free digraph of order $n-1$, by the induction hypothesis we have
\begin{eqnarray}\label{eq342}
\nonumber f(A(i))&\le& ex(n-1,\mathscr{F}_{k})\\
\nonumber &= &{n-1 \choose 2} -{s\choose 2}k-st\\ &=&\frac{(n-1)(n-2)}{2}-\frac{(s-1)(n-1)}{2}-\frac{(s+1)t}{2}.
\end{eqnarray}
Applying Lemma \ref{le3} we have
\begin{eqnarray*}
f(A)&\le& \frac{n(n-1)}{2}-\frac{(s-1)(n-1)}{2}-\frac{(s+1)t}{2}-s\\
 &=&\frac{n(n-1)}{2}-\frac{(s-1)n}{2}-\frac{(s+1)(t+1)}{2}\\
 &=&\frac{n(n-1)}{2}-\frac{(u-1)n}{2}-\frac{(u+1)v}{2}.
  \end{eqnarray*}
 Hence,
 \begin{eqnarray*}
ex(n,\mathscr{F}_{k})\le
  \frac{n(n-1)}{2}-\frac{(u-1)n}{2}-\frac{(u+1)v}{2}.
  \end{eqnarray*}

On the other hand, if $D$ is a   $(u,k,v)$-completely transitive tournament, then there exist  $k-t$ numbers $j_1,j_2,\ldots,j_{k-t}\in \{uk+1,uk+2,\ldots,(u+1)k\}$ such that
$$PAP^T=\Pi_{u+1,k}(j_1,j_2,\ldots,j_{k-v})$$ for some permutation matrix $P$.
Since $(\Pi_{u+1,k})^k=0$, we have $(A_D)^k=0$, and hence $D$ is $\mathscr{F}_{k}$-free. Moreover,
  the size of $D$ is
$$f(A)=\frac{n(n-1)}{2}-\frac{(u-1)n}{2}-\frac{(u+1)v}{2}.$$
Hence,
 \begin{eqnarray}\label{eq43}
ex(n,\mathscr{F}_{k})=
  \frac{n(n-1)}{2}-\frac{(u-1)n}{2}-\frac{(u+1)v}{2}={n \choose 2}-{u \choose 2}k-uv
  \end{eqnarray}
   and any $(u,k,v)$-completely transitive tournament is in $Ex(n,\mathscr{F}_{k})$.
 \\
\par
Conversely, suppose $D\in Ex(n,\mathscr{F}_{k})$ and $A\equiv A_D$.
Applying Lemma \ref{le3} to $A$, by (\ref{eq43}) we know there is some $i_0\in \{1,2,\ldots,n\}$ such that equality in (\ref{eq342}) holds. Without loss of generality, we assume $i_0=n$.  We distinguish two cases.

{\it Case 1.} $t=0$. Then $u=s\ge 2$ and $v=1$. By the induction hypothesis we may assume
\begin{equation}\label{eq344}
A=\begin{bmatrix}
T_{k}&T_{k}&\cdots&T_{k}&x_1\\
T_{k}&T_{k}&\cdots&T_{k}&x_2\\
\vdots&\vdots&\ddots&\vdots&\vdots\\
T_{k}&T_{k}&\cdots&T_{k}&x_s\\
y_1^T&y_2^T&\ldots&y_s^T&a_{n,n}
\end{bmatrix}\in M_{sk+1}\{0,1\},
\end{equation}where $x_{i},y_{i}\in \mathbb{R}^{k}$ for $i=1,2,\cdots,s$.

 Let $a_{s_{i},n}$ be the last nonzero component in $x_{i}$,  and $a_{n,t_{i}}$ be the first nonzero component in $y_{i}$  for $i=1,\ldots, s$. Here we define $s_i=(i-1)k$ when $x_{i}=0$, and $t_{i}=ik+1$  when $y_i=0$.

 We claim
 \begin{equation}\label{eq44}
 t_{i}\ne s_{i}+1 \textrm{~~for~~} 1\le i\le s.
 \end{equation}
  Note that we can change the role of each $(x_i,y_i)$ by permutation similarity. To prove (\ref{eq44}) it suffices to verify the case $i=1$.
  Suppose $t_1=s_1+1$.
  If  $s_{1}\le 2$, then $D$ has two distinct walks of length $k$ from 1 or $n$ to $k$:
    $$\left\{\begin{array}{l}
  1\rightarrow \cdots\rightarrow s_{1}\rightarrow n\rightarrow s_{1}+1\rightarrow s_{1}+2\rightarrow s_{1}+3\rightarrow  \cdots \rightarrow k,\\
 1\rightarrow \cdots\rightarrow s_{1}\rightarrow n\rightarrow s_{1}+1\rightarrow k+s_{1}+2\rightarrow s_{1}+3\rightarrow \cdots\rightarrow  k,
\end{array}\right.$$ where the walk $1\rightarrow \cdots\rightarrow s_{1}$ does not appear when $s_1=0$.
If $3\le s_{1}\le k$, then $D$ has two distinct walks of length $k$ from 1 to $n$ or $k$:
  $$\left\{\begin{array}{l}
 1\rightarrow  2\rightarrow  3\rightarrow \cdots\rightarrow s_{1}\rightarrow n\rightarrow s_{1}+1\rightarrow s_{1}+2\rightarrow \cdots\rightarrow k,\\
  1\rightarrow  k+2\rightarrow  3\rightarrow \cdots\rightarrow s_{1}\rightarrow n\rightarrow s_{1}+1\rightarrow s_{1}+2\rightarrow \cdots\rightarrow  k,
\end{array}\right.$$
where the walk  $n\rightarrow s_{1}+1\rightarrow s_{1}+2\rightarrow \cdots\rightarrow  k$ does not appear when $s_1=k$. This contradicts $D\in Ex(n,\F)$ and (\ref{eq44}) follows.

Denote by $$B_i=\begin{bmatrix}T_k&x_i\\y_i^T&a_{nn}\end{bmatrix} \textrm{~~for~~} 1\le i\le s.$$ Given any $i\in \{1,2,\ldots,s\}$, since $B_i$ is a principal submatrix of $A$, then $B_i\in \Gamma(k+1,k)$. By Theorem \ref{thh1} we have $f(B_i)\le k(k+1)/2$ and
\begin{equation}\label{eq45}
f(x_i)+f(y_i)+a_{nn}=f(B_i)-f(T_k)\le k.
\end{equation}
If equality in (\ref{eq45}) holds, then applying Lemma \ref{le7} to $B_i$ we get $t_i=s_i+1$, which contradicts (\ref{eq44}). Therefore, we have
\begin{equation*}
f(x_i)+f(y_i)+a_{nn}\le k-1 \textrm{~~for~~} 1\le i\le s.
\end{equation*}

Since
\begin{eqnarray*}
n-s-1= s(k-1)&\ge& \sum_{i=1}^s[f(x_i)+f(y_i)+a_{nn}]\\
&\ge& \sum_{i=1}^s[f(x_i)+f(y_i)]+a_{nn}\\
& =&f(A)-s^2f(T_k)=n-s-1
\end{eqnarray*}
we get $a_{nn}=0$ and
\begin{equation*}
f(x_i)+f(y_i)= k-1 \textrm{~~for~~} 1\le i\le s.
\end{equation*}
By (\ref{eq44}), applying Lemma \ref{le7} to each $B_i$ again we have
 $$t_i=s_i+2 \textrm{~~and~~}
  x_i=(J_{1, q_i},0)^T,~~y_i=(0,J_{1, k-q_i-1})^T \textrm{~~with~~} q_i\in\{0,1,\ldots,k-1\}$$
for all $1\le i\le s$.
Now we assert $$q_1=q_2=\cdots=q_s=q \textrm{~~for some~~} q\in\{0,1,2,\ldots,k-1\}.$$  Otherwise, without loss of generality we assume $q_1< q_2$. Then
$$s_1=q_1,  s_2=k+q_2.$$ Denote by $P$ the permutation matrix obtained by  interchanging row $q_1+1$ and row $k+q_1+1$ of the $n\times n$ identity matrix. Let $A'=(a'_{ij})=PAP^T$.
Then $A'\in Ex(n,\F)$ has the same form (\ref{eq344}) as $A$. Moreover, the last nonzero component in $x_1$ is $a'_{q_1+1,n}=a_{k+q_1+1,n}$; the first nonzero component in $y_1$ is  $a'_{n,q_1+2}=a_{n,q_1+2}$. If we define $s'_i$ and $t'_i$ for $A'$ the same as $s_i$ and $t_i$ for $A$, then
$$s'_1=q_1+1,  \textrm{~~and~~}
  t'_1=    q_1+2.$$
On the other hand, using the same arguments as above we have $t'_1=s'_1+2$, a contradiction.

Therefore, we have $x_{i}=(J_{1,q},0)^T$, $y_{i}=(0,J_{1,k-q-1})^T$ for $i=  1,\ldots,s$, and  $A$ is permutation similar to $\Pi_{s+1 ,k}(\beta)$ with $\beta=\{sk+1,sk+2,\ldots,sk+k\}\setminus\{q+1\}$. Hence, $D$ is a $(u,k,v)$-completely transitive tournament.

Case 2. $t\ne 0$. By the induction hypothesis we may assume
  $$A=\begin{bmatrix}
T_{k}&T_{k}&\cdots&T_{k}&w_1&w_2&\cdots&w_{t}&x_1\\
T_{k}&T_{k}&\cdots&T_{k}&w_1&w_2&\cdots&w_{t}&x_2\\
\vdots&\vdots&&\vdots&\vdots&\vdots&&\vdots&\vdots\\
T_{k}&T_{k}&\cdots&T_{k}&w_1&w_2&\cdots&w_{t}&x_{s}\\
 u_1&u_1&\cdots&u_1&0&1&\cdots&1&a_{sk+1,n}\\
u_2&u_2&\cdots&u_2&0&0&\ddots&\vdots&\vdots\\
\vdots&\vdots&&\vdots&\vdots&\vdots&\ddots&1&a_{sk+t-1,n}\\
u_{t}&u_{t}&\cdots&u_{t}&0&0&\cdots&0&a_{sk+t,n}\\
y_1^T&y_2^T&\ldots&y_{s}^T&a_{n,sk+1}&a_{n,sk+2}&\cdots&a_{n,sk+t}&a_{nn}
\end{bmatrix}\in M_{s k+t+1}\{0,1\},$$where $x_{i},y_{i}\in \mathbb{R}^{k}$  for $i=1,2,\ldots,s$, $w_j=(J_{1,k_j},0)^T$, $u_j=(0,J_{1,k-1-k_j})$, and $0\leq k_1<k_2<\cdots<k_{t}\leq k-1$.

Let $i\in\{1,2,\ldots,t\}$. We claim
\begin{equation}\label{eq47}
a_{sk+i,n}+a_{n,sk+i}=1.
\end{equation}
Since $A(sk+i)\in \Gamma(n-1,k)$, we have
\begin{eqnarray}\label{eq348}
a_{sk+i,n}+a_{n,sk+i}&=&f(A)-f(A(sk+i))-  s[f(w_i)+f(u_i)]-
(t-1)  \nonumber\\
&\ge& f(A)-ex(n-1,\F)- s(k-1)-(t-1)=1.
\end{eqnarray}
If  $a_{sk+i,n}=a_{n,sk+i}=1$, setting $$\alpha=\{1,2,\ldots,k,sk+i,n\}$$
 and applying Lemma \ref{le10} to $A[\alpha]$ we get $A[\alpha]\not\in \Gamma(k+2,k)$, which contradicts $A\in \Gamma(n,k)$.
Thus we get (\ref{eq47}).

Combining (\ref{eq47}) and (\ref{eq348}) we get
 $$f(A(sk+i))=ex(n-1,\F).$$ By the induction hypothesis, $A(sk+i)$ is permutation similar to
  a submatrix of $\Pi_{s+1,k}$. Therefore, $a_{nn}=0$ \textrm{  and } $$ \sum_{i=1}^s[f(x_i)+f(y_i)]=f(A)-f(A(n))-a_{nn}-
  \sum_{i=1}^t(a_{sk+i,n}+a_{n,sk+i})=s(k-1).$$

Next we distinguish two subcases.

{\it Subcase 1.}    $s>1$. Let $\alpha
=\{sk+1,sk+2\ldots,sk+t\}$. We consider $A(\alpha)$. Then
\begin{eqnarray*}
f(A(\alpha))&=&f(A)-s\sum_{i=1}^t[f(w_i)+f(u_i)]-t(t-1)/2-\sum_{i=1}^t(a_{sk+i,n}+a_{n,sk+i})-a_{nn}\\
&=&\frac{n(n-1)}{2}-\frac{(s-1)n}{2}-\frac{(s+1)(t+1)}{2}-st(k-1)-t(t-1)/2-t\\
&=& \frac{(n-t)(n-t-1)}{2}-\frac{(s-1)(n-t)}{2}-\frac{s+1}{2}\\
&=&ex(n-t,\F).
\end{eqnarray*}
 Applying  Case 1 to $A(\alpha)$, we have $$x_j=(J_{1,q},0)^T  \textrm{~~and~~} y^T_j=(0,J_{1,k-1-q})^T \textrm{~~for all~~}j\in \{1,\ldots,s\} $$ with $  0\le q\le k-1$.

  We assert   $q\not\in \{k_1,k_2,\ldots,k_t\}$. Otherwise suppose $q=k_i$ for some $i$.  Since $$a_{q+1,n}=a_{n,q+1}=a_{q+1,sk+i}=a_{sk+i,q+1}=0,$$  we have
   $$\delta_{q+1}\le s(k-1)+t-1=n-s-2$$
    and
    $$f(A(q+1))=f(A)-\delta_{q+1}>ex(n-1,\F),$$ which contradicts  $A(q+1)\in \Gamma(n-1,k)$.

Next we show that
\begin{equation*}\label{eq49}
a_{sk+i,n}=1 \textrm{~~when~~}q>k_i \textrm{~~for~~} i=1,\ldots,t.
\end{equation*}
 Otherwise suppose $q>k_i$ and $a_{sk+i,n}=0$. Then by (\ref{eq47}) we have $a_{sk+i,q+1}=a_{n,sk+i}=1$.
If $q\le 2$, then  $D$ has two distinct walks of length $k$ from $1$ to $k$:
$$\left\{\begin{array}{l}
  1\rightarrow \cdots\rightarrow q\rightarrow n\rightarrow sk+i\rightarrow q+1\rightarrow  q+2\rightarrow q+3\rightarrow \cdots\rightarrow k-1\rightarrow k,\\
  1\rightarrow \cdots\rightarrow q\rightarrow n\rightarrow sk+i\rightarrow q+1\rightarrow k+q+2\rightarrow q+3\rightarrow \cdots\rightarrow k,
\end{array}\right.$$
a contradiction.
If $q\ge 3$, then $D$ has two distinct walks of length $k$ from $1$ to $k$:
$$\left\{\begin{array}{l}
  1\rightarrow 3\rightarrow   \cdots\rightarrow q\rightarrow n\rightarrow sk+i\rightarrow q+1\rightarrow \cdots\rightarrow k-1\rightarrow k,\\
  1\rightarrow 2\rightarrow  \cdots\rightarrow q-1\rightarrow n\rightarrow sk+i\rightarrow q+1\rightarrow \cdots\rightarrow k-1\rightarrow k,
\end{array}\right.$$
a contradiction.

Using similar arguments as above, we can deduce $a_{sk+i,n}=0$ when $q<k_i$.
Let $$\beta=\{sk+1,sk+2,\ldots,n\}\setminus\{sk+k_1+1,sk+k_2+1,\ldots,sk+k_t+1,sk+q+1\}.$$ If $q>k_t$, then $A=\Pi_{s+1,k}(\beta)$. Otherwise let
\begin{eqnarray*}
P=\left\{\begin{array}{ll}
 \begin{bmatrix}
0&1\\
I_{t}&0\end{bmatrix},&\textrm{ if }q<k_1,\\
 \begin{bmatrix}
 I_{j}&0&0\\
 0&0&1\\
0&I_{t-j}&0\end{bmatrix},&\textrm{ if } k_j<q<k_j+1,
\end{array}\right.
\end{eqnarray*}
 and $Q=I_{sk}\oplus P$. Then $QAQ^T=\Pi_{s+1,k}(\beta)$. Therefore, $A$ is  permutation similar to  $\Pi_{s+1,k}(\beta)$  and $D$ is a  $(u,k,v)$-completely transitive tournament.

{\it Subcase  2.} $s=1$.   Then $t\ge 5$ and
$$A=(a_{ij})=\begin{bmatrix}
T_{k}&w_{1}&w_{2}&\cdots&w_{t}&x\\
u_{1}&0&1&\cdots&1&a_{k+1,n}\\
u_{2}&0&0&\ddots&\vdots&\vdots\\
\vdots&0&0&\ddots&1&\vdots\\
u_{t}&0&0&0&0&a_{k+t,n}\\
y^{T}&a_{n,k+1}&a_{n,k+2}&\cdots&a_{n,k+t}&0
\end{bmatrix}\in M_{n}\{0,1\},$$ where $x,y\in \mathbb{R}^{n-5}$, $w_{i}=(J_{1,k_i},0)^T$, $u_{i}=(0,J_{1,k-1-k_i})$, and $0\leq k_1<k_2<\cdots<k_t\leq k-1$.

Choose any three distinct numbers $p,q,r\in \{1,2,\ldots,t\}$.
Denote $\alpha=\{1,2,\ldots,k,k+p, k+q,k+r,n\}$ and $B=A[\alpha]$. Then
$$f(B)=f(T_k)+3(k-1)+3+f(x)+f(y)+3=(k+4)(k+3)/2-4=ex(k+4,\F).$$
  Applying Corollary \ref{co1} to $B$, we have $$x=(J_{1,h},0)^T, ~~y=(0,J_{1,k-h-l})^T,$$  where $0\le h\le k-1$. Moreover, we have $h\not\in\{k_p,k_q,k_r\}$ and
  \begin{equation}\label{eq50}
  ~~a_{k+i,n}=\left\{\begin{array}{ll}
  1,& \textrm{ if } h>k_i,  \\
  0,& \textrm{ if } h<k_i,
  \end{array}
  \right.\end{equation} \textrm{ for } $i=p,q,r$.

Since $p,q,r$ are arbitrarily chosen from $\{1,2,\ldots,t\}$, we have
$$h\not\in\{k_1,k_2,\ldots,k_t\}$$ and
(\ref{eq50}) holds  \textrm{ for } $i=1,2,\ldots,t.$  Using the same arguments as in the previous subcase, we can prove that
  $D$ is  a $(u,k,v)$-completely transitive tournament.

This completes the proof.     \hspace{10cm}                  $\Box$

\section{Further discussion}

In this section we discuss the unsolved cases for Problem \ref{pro1}. We focus our attention on  strict digraphs.
The second part of Theorem \ref{thh2} may  not be true for $n=k+5$ when $k=4$. For example, let $D$ be the digraph with adjacency matrix
$$A=\begin{bmatrix}
0&0&1&1&1&1&1&1&1\\
0&0&1&1&1&1&1&1&1\\
0&0&0&1&1&1&1&0&1\\
0&0&0&0&1&1&1&0&0\\
0&0&0&0&0&1&1&0&0\\
0&0&0&0&0&0&0&0&0\\
0&0&0&0&0&0&0&0&0\\
0&0&0&0&1&1&1&0&1\\
0&0&0&0&0&1&1&0&0
\end{bmatrix}.$$
Then $f(A)=30$ and $A^4\in M_{9}\{0,1\}$.  But $A$ is not permutation similar to any principal submatrix of $\Pi_{4,3}$, since $A^4\ne 0$.
However, when $k=4$ and $n$ is sufficiently large, we  conjecture  the second part of Theorem \ref{thh2} is still true.

When $k=3$, Theorem \ref{thh2} does not hold, which  is shown by the following example.  Let $$T'_3=\begin{bmatrix}
0&1&1\\
0&1&1\\
0&0&0
\end{bmatrix} $$  and
$$A=\begin{bmatrix}
T_{3}&T'_{3}&T_{3}&\cdots&T_{3}\\
T_{3}&T_{3}&T_{3}&\cdots&T_{3}\\
T_{3}&T_{3}&T_{3}&\cdots&T_{3}\\
\vdots&\vdots&\vdots&\ddots&\vdots\\
T_{3}&T_{3}&T_{3}&\cdots&T_{3}
\end{bmatrix}\in M_{3t}\{0,1\}.$$
Denote $(b_{ij})=A^3$. Then $$b_{ij}=
\begin{cases}
1,&   \textrm{if~~} i=3s+1,j=3l \textrm{~~for~~}  s\in\{0,\ldots,t-1\}, l\in\{1,\ldots,t\};\\
0,& \textrm{otherwise}.
\end{cases}
$$
Therefore, $D(A)$ is an $\mathscr{F}_3$-free digraph and
$$ex(3t,\mathscr{F}_3)\ge f(A)>f(\Pi_{t,3}).$$

When $k=2$, by  \cite[Theorem 2]{WU} we can deduce
$$
ex(n,\mathscr{F}_2)=\begin{cases}\frac{n^2+4n-5}{4}, & \text{if $n$ is odd,}\\
    \frac{n^2+4n-8}{4},& \text{if $n$ is even and $n\not=4,$}\\
    7, & \text{if $n=4$}.\end{cases}
$$
 The set of extremal digraphs $Ex(n,\mathscr{F}_2)$ is not known.

We leave the problems of determining $ex(n,\mathscr{F}_3)$ and $Ex(n,\mathscr{F}_k)$ for $k=2,3,4$ for  future work.

\section*{Acknowledgement}

The first author is grateful to Professor Yuejian Peng for helpful discussions on  extremal graph theory.  The research of Huang
was supported by the NSFC grant 11401197 and a Fundamental Research
Fund for the Central Universities.

\end{document}